\documentclass [twoside,reqno, 12pt] {amsart}

\usepackage{hyperref}

\usepackage{amsfonts}
\usepackage{amssymb}
\usepackage{a4}
\usepackage{color}

\newtheorem{thm}{Theorem}[section]
\newtheorem{cor}[thm]{Corollary}
\newtheorem{lem}[thm]{Lemma}
\newtheorem{prop}[thm]{Proposition}

\newtheorem{defn}[thm]{Definition}
\newtheorem{rem}[thm]{Remark}

\theoremstyle{definition}

\numberwithin{equation}{section}

\newcommand{\C}{\mathbb{C}}

\renewcommand{\div}{\operatorname{div}}

\newcommand{\N}{\mathbb{N}}

\newcommand{\R}{\mathbb{R}}

\newcommand{\supp}{\operatorname{supp}}

\def\hat{\widehat}
\def\tilde{\widetilde}
\def \bfo {\begin {eqnarray*} }
\def \efo {\end {eqnarray*} }
\def \ba {\begin {eqnarray*} }
\def \ea {\end {eqnarray*} }
\def \beq {\begin {eqnarray}}
\def \eeq {\end {eqnarray}}
\def \supp {\hbox{supp }}

\def \det {\hbox{det}}

\def \e {\varepsilon}
\def \p {\partial}

\def\hat{\widehat}
\def\tilde{\widetilde}
\def \bfo {\begin {eqnarray*} }
\def \efo {\end {eqnarray*} }
\def \ba {\begin {eqnarray*} }
\def \ea {\end {eqnarray*} }
\def \beq {\begin {eqnarray}}
\def \eeq {\end {eqnarray}}
\def \supp {\hbox{supp }}

\def \det {\hbox{det}}

\def \e {\varepsilon}
\def \p {\partial}

\begin{document}

 \title[Inverse problems for nonlinear magnetic Schr\"odinger equations ]{Inverse problems for nonlinear magnetic Schr\"odinger equations on conformally transversally anisotropic manifolds}

\author[Krupchyk]{Katya Krupchyk}

\address
        {K. Krupchyk, Department of Mathematics\\
University of California, Irvine\\
CA 92697-3875, USA }

\email{katya.krupchyk@uci.edu}

\author[Uhlmann]{Gunther Uhlmann}

\address
       {G. Uhlmann, Department of Mathematics\\
       University of Washington\\
       Seattle, WA  98195-4350\\
       USA\\
        and Institute for Advanced Study of the Hong Kong University of Science and Technology}
\email{gunther@math.washington.edu}

\maketitle

\begin{abstract}

We study the inverse boundary problem for a nonlinear magnetic Schr\"odinger operator on a conformally transversally anisotropic Riemannian manifold of dimension $n\ge 3$. Under suitable assumptions on the nonlinearity,  we show that the knowledge of the Dirichlet-to-Neumann map on the boundary of the manifold determines the nonlinear magnetic and electric potentials uniquely. No assumptions on the transversal manifold are made in this result, whereas the corresponding inverse boundary problem for the linear magnetic Schr\"odinger operator is still open in this generality.

\end{abstract}

\section{Introduction and statement of results}

Let $(M,g)$ be a smooth compact oriented Riemannian manifold of dimension $n\ge 3$ with smooth boundary.  Let $A\in C^\infty(M, T^*M)$ be a $1$-form with complex valued $C^\infty$ coefficients, and let 
\[
d_A=d+i A: C^\infty(M)\to C^\infty(M, T^*M),
\]
where $d: C^\infty(M)\to C^\infty(M, T^*M)$ is the de Rham differential. The formal $L^2$--adjoint $d_A^*: C^\infty(M, T^*M)\to C^\infty(M)$ of $d_A$ is defined by
\[
(d_A u, v)_{L^2(M,T^*M)}=(u,d^*_Av)_{L^2(M)}, \quad u\in C_0^\infty(M^\text{int}), \quad v\in C_0^\infty(M^{\text{int}}, T^*M^{\text{int}}),
\] 
where $M^{\text{int}}=M\setminus\p M$ stands for the interior of $M$. Here and in what follows, when $u,v\in C^\infty(M)$, we write 
\[
(u,v)_{L^2(M)}=\int_M u\overline{v}dV_g,
\] 
for the natural $L^2$--scalar product, where $dV_g$ is the Riemannian volume element on $M$.   Similarly, when $\alpha,\beta\in C^\infty(M, T^*M)$ are $1$-forms, we define the $L^2$--scalar product,
\[
( \alpha, \beta)_{L^2(M, T^*M)}=\int_M  \langle \alpha, \overline{\beta}\rangle_g dV_g(x),
\]
where $\langle\cdot, \cdot\rangle_g$ is the pointwise scalar product in the space of $1$--forms induced by the Riemannian metric $g$. In local coordinates $(x_1,\dots, x_n)$ in which $\alpha=\sum_{j=1}^n\alpha_j dx_j$, $\beta=\sum_{j=1}^n\beta_j dx_j$ and $(g^{jk})$ is the matrix inverse of 
$(g_{jk})$, $g=\sum_{j,k=1}^n g_{jk}dx_jdx_k$, we have
\[
\langle \alpha, \beta\rangle_g=\sum_{j,k=1}^n g^{jk}\alpha_j \beta_k. 
\]
We also have
\[
d_A^*=d^*-i \langle\overline{A}, \cdot\rangle_g. 
\]
In local coordinates,   we see that 
\begin{equation}
\label{eq_int_2}
d^*v=-\sum_{j,k=1}^n|g|^{-\frac{1}{2}}\p_{x_j}(|g|^{\frac{1}{2}}g^{jk}v_k),
\end{equation}
where $|g|=\det(g_{jk})$ and $v=\sum_{j=1}^n v_jdx_j$.

In this paper we shall consider $1$-forms and scalar functions depending holomorphically on a parameter $z\in \C$. Specifically, let $A: M\times \C\to T^*M$ and  $V: M\times \C\mapsto \C$ satisfy the conditions: 
\begin{itemize}
\item[($A_i$)] the map $\C\ni z\mapsto A(\cdot, z)$ is holomorphic with values in $C^{1,1}(M,T^*M)$, 
the space of $1$-forms with complex valued $C^{1,1}(M)$ coefficients, 
\item[($V_{i}$)] the map  $\C\ni z\mapsto V(\cdot, z)$ is holomorphic with values in $C^{1,1}(M)$, 
\item[($V_{ii}$)] $V(x,0)=0, \text{ for all } x\in M$. 
\end{itemize}
Here $C^{1,1}(M)$ is the space of $C^1$ functions on $M$ with a Lipschitz gradient.

It follows from ($A_i$), ($V_i$), ($V_{ii}$)  that $A$ and $V$ can be expanded into  power series 
\begin{equation}
\label{eq_int_2_A}
A(x,z)=\sum_{k=0}^\infty A_k(x)\frac{z^k}{k!}, \quad A_k(x):=\p_z^k A(x,0)\in C^{1,1}(M,T^*M),
\end{equation}
converging in the $C^{1,1}(M,T^*M)$ topology, and 
\begin{equation}
\label{eq_int_2_V}
V(x,z)=\sum_{k=1}^\infty V_k(x)\frac{z^k}{k!}, \quad V_k(x):=\p_z^k V(x,0)\in C^{1,1}(M),
\end{equation}
converging in the $C^{1,1}(M)$ topology. 

Let us introduce the nonlinear magnetic Schr\"odinger operator defined by 
\begin{equation}
\label{eq_int_3_L_AV}
\begin{aligned}
L_{A,V}u&=d^*_{\overline{A(\cdot, u)}}d_{A(\cdot, u)}u+V(\cdot, u)\\
=-\Delta_gu&+d^*(i A(\cdot, u)u)-i\langle A(\cdot, u), du\rangle_g+ \langle A(\cdot, u), A(\cdot, u)\rangle_g u+V(\cdot, u),
\end{aligned}
\end{equation} 
for $u\in C^\infty(M)$.  Notice that the first order linearization of $L_{A,V}$ is the standard linear magnetic Schr\"odinger operator $d^*_{\overline{A_0}}d_{A_0}+V_1$. 

Furthermore, we also assume that $A_0\in C^\infty(M,T^*M)$, $V_1\in C^\infty(M)$, and that 
\begin{itemize}
\item[(i)] $0$ is not a Dirichlet eigenvalue of the operator $d^*_{\overline{A_0}}d_{A_0}+V_1$.
\end{itemize}

Consider the Dirichlet problem for the nonlinear magnetic Schr\"odinger operator, 
\begin{equation}
\label{eq_int_3}
\begin{cases} L_{A,V}u=0 & \text{in}\quad M^{\text{int}}, \\
u|_{\p M}=f.
\end{cases}
\end{equation}
It is shown in Theorem \ref{thm_well-posedness} that under the above assumptions, there exist $\delta>0$ and $C>0$ such that when $f\in B_\delta(\p M):=\{f\in C^{2,\alpha}(\p M): \|f\|_{C^{2,\alpha}(\p M)}<\delta\}$, $0<\alpha<1$, the problem \eqref{eq_int_3} has a unique solution $u=u_f\in C^{2,\alpha}(M)$ satisfying 
$\|u\|_{C^{2,\alpha}(M)}<C\delta$. Here $C^{2,\alpha}(M)$ stands for the standard H\"older space of functions on $M$.  Associated to the problem \eqref{eq_int_3}, we define the Dirichlet--to--Neumann map 
\begin{equation}
\label{eq_int_3_DN}
\Lambda_{A,V}f=\p_\nu u_f|_{\p M},
\end{equation}
where $f\in B_\delta(\p M)$ and $\nu$ is the unit outer normal to the boundary.

The inverse problem that we are interested in is whether the knowledge of the Dirichlet--to--Neumann map $\Lambda_{A,V}$ determines the nonlinear magnetic and electric potentials, $A$ and $V$, respectively. 

When $A=0$ and $V(x,z)=V_1(x)z$, the inverse problem for the linear Sch\"odinger operator $-\Delta_g+V_1$ is related to the Calder\'on problem, which has been the object of intense studies but remains open in the case of a general smooth Riemannian manifold $(M,g)$ of dimension $n\ge 3$ with smooth boundary.  Let us mention that the unique determination of the potential $V_1$ from the knowledge of the Dirichlet--to--Neumann map $\Lambda_{0, V_1}$ was established in \cite{Sylvester_Uhlmann_1987} in the Euclidean setting, in \cite{Isozaki_2004} for hyperbolic manifolds, and in \cite{Lee_Uhlmann}, \cite{Lassas_Uhlmann}, \cite{Kohn_Vogelius_1984} in the analytic case. The uniqueness in the inverse boundary problem for the linear magnetic Schr\"odinger operator $d^*_{\overline{A_0}}d_{A_0}+V_1$ up to a suitable gauge transformation was obtained in \cite{Nakamura_Sun_Uhlmann} in the Euclidean setting, see also \cite{Krup_Uhlmann_2014}. Going beyond these settings, the most general uniqueness results were obtained in the case when the manifold $(M,g)$ is conformally transversally anisotropic and the transversal manifold satisfies some additional assumptions. Following \cite{DKSaloU_2009} and \cite{DKurylevLS_2016}, let us recall the definition of a conformally transversally anisotropic manifold. 

\begin{defn}
A compact smooth oriented Riemannian manifold $(M,g)$ of dimension $n\ge 3$ with smooth boundary is said to be conformally transversally anisotropic  if  there exists an $(n-1)$--dimensional smooth compact Riemannian manifold $(M_0,g_0)$ with smooth boundary such that $M\subset\subset  \R\times M_0$  and $g=c(e\oplus g_0)$ where $e$ is the Euclidean metric on $\R$ and $c$ is a positive smooth function on $M$. 

\end{defn}

In the case when $(M,g)$ is conformally transversally anisotropic, assuming that the transversal manifold $(M_0,g_0)$ is simple in the sense that the boundary $\p M_0$ is strictly convex and for any point $p\in M_0$, the exponential map $\exp_p$ with its maximal domain of definition in $T_pM_0$ is a diffeomorphism onto $M_0$, the global uniqueness for the inverse boundary problem for the linear magnetic Schr\"odinger equation up to a gauge was proven in \cite{DKSaloU_2009}, see also \cite{Krup_Uhlmann_magn_2018}.  Note that the geodesic ray transform on functions and $1$-forms is invertible on simple manifolds, see  \cite{Anikonov_78}, \cite{Muhometov}. 

These uniqueness results were strengthened in \cite{DKurylevLS_2016}, where the global uniqueness in the inverse boundary problem for the linear Schr\"odinger  equation was established under the assumption that the geodesic ray transform on the transversal manifold is injective. Similar results for the inverse boundary problem for the linear magnetic Schr\"odinger equation were obtained in \cite{Cekic},  \cite{Krup_Uhlmann_magn_2018}. The injectivity of the geodesic ray transform is open in general, and has only been established under certain geometric assumptions. In particular, the injectivity of the geodesic ray transform is proven in \cite{Stefanov_Uhlmann_Vasy}, \cite{Uhlmann_Vasy_2016} when $M_0$ has strictly convex boundary and is foliated by strictly convex hypersurfaces, and in \cite{Guillarmou_2017}, \cite{Guill_Mazz_Tzou} when $M_0$ has a hyperbolic trapped set and no conjugate points. As an example of the latter, one can consider a negatively curved manifold $M_0$. We refer to \cite{DKuLLS_2018} where the
linearized anisotropic Calder\'on problem was studied on a transversally anisotropic manifold under certain mild conditions on the transversal manifold related to the geometry of pairs of intersecting geodesics.

Turning the attention to inverse  problems for nonlinear PDE, it was discovered in \cite{Kurylev_Lassas_Uhlmann_2018} that nonlinearity can be helpful in solving inverse problems for hyperbolic equations, see also \cite{Lassas_Uhlmann_Wang}, \cite{Feizmohammadi_Lassas_Oksanen}, and the references given there. Similar phenomena for inverse problems for semilinear elliptic PDE have been revealed in  \cite{Feizmohammadi_Oksanen}, \cite{LLLS}, see also \cite{LLLS_partial}, \cite{Krup_Uhlmann_non_linear_1}, \cite{Krup_Uhlmann_2}, \cite{Lai_Ting}.  A common feature of all of the aforementioned works is that the presence of a nonlinearity allows one to solve inverse problems for nonlinear equations in cases where the corresponding inverse problem in the linear setting is open.

In particular, the inverse boundary problem for the nonlinear Schr\"odinger equation $L_{0, V}u=-\Delta_gu+ V(\cdot, u)=0$ on  a conformally transversally anisotropic manifold $(M,g)$ of dimension $n\ge 3$ was studied in \cite{Feizmohammadi_Oksanen}, \cite{LLLS}, and the following result was obtained: if $V$ satisfies the assumptions ($V_i$),  ($V_{ii}$), and 
\begin{itemize}
\item[($V_{iii}$)] $\p_z V(x,0)=\p_z^2 V(x,0)=0, \text{ for all } x\in M$,
\end{itemize}
then the knowledge of the Dirichlet--to--Neumann map $\Lambda_{0,V}$ determines  $V$ in $M\times \C$ uniquely. Notice that remarkably there are no assumptions on the transversal manifold in this result while the inverse problem for the linear Schr\"odinger equation is still open in this generality. The proof of this result relies on higher order linearizations of the Dirichlet--to--Neumann map, which allow one to reduce the inverse problem to the following density result, see \cite{LLLS},

\begin{prop}
\label{prop_density_potential}
Let $(M,g)$ be a conformally transversally anisotropic manifold of dimension $n\ge 3$, and  let $q\in C^{1,1}(M)$. If 
\begin{equation}
\label{eq_2_1_potential}
\int_M q u_1u_2u_3 u_4 dV_g=0,
\end{equation}
for all harmonic functions $u_j\in C^\infty(M)$, $j=1,2,3,4$,  then $q\equiv 0$. 
\end{prop}

The purpose of this paper is to extend the aforementioned result of \cite{Feizmohammadi_Oksanen}, \cite{LLLS} to the nonlinear magnetic Schr\"odinger equation $L_{A,V}u=0$ given by \eqref{eq_int_3_L_AV}. To state our result, similarly to the assumption ($V_{iii}$) on the potential $V$, we shall also suppose that the nonlinear magnetic potential $A$ satisfies, 

\begin{itemize}
\item[($A_{ii}$)] $A(x,0)=\p_z A(x,0)=0, \text{ for all } x\in M$. 
\end{itemize}

Our main result is as follows. 
\begin{thm}
\label{thm_main} Let  $(M,g)$ be a conformally transversally anisotropic manifold of dimension $n\ge 3$. Let $A^{(1)}, A^{(2)} : M\times \C\to T^*M$ and $V^{(1)}, V^{(2)}: M\times \C\mapsto \C$  satisfy the assumptions $(A_i)$, $(A_{ii})$  and $(V_i)$, $(V_{ii})$, $(V_{iii})$,  respectively.  
If $\Lambda_{A^{(1)},V^{(1)}}=\Lambda_{A^{(2)},V^{(2)}}$ then $A^{(1)}=A^{(2)}$ and $V^{(1)}=V^{(2)}$ in $M\times \C$. 
\end{thm}

\begin{rem} Let us point out that there are no assumptions on the transversal manifold in Theorem \ref{thm_main}, whereas the corresponding inverse boundary problem for the linear magnetic Schr\"odinger operator is still open in this generality.
\end{rem}

\begin{rem}
Notice that as opposed to the inverse boundary problem for the linear magnetic Schr\"odinger equation, where one can determine the magnetic potential up to a gauge transformation only, in our nonlinear setting the unique determination of both potentials is possible, due to the assumptions $(A_i)$, $(A_{ii})$  and $(V_i)$, $(V_{ii})$, $(V_{iii})$, which imply that the first order linearization of the nonlinear equation is given by $-\Delta_g u=0$, rather than by the linear magnetic Schr\"odinger equation. 
\end{rem}

Similarly to  \cite{Feizmohammadi_Oksanen}, \cite{LLLS}, the proof of Theorem \ref{thm_main} relies on higher order linearizations of the Dirichlet--to--Neumann map $\Lambda_{A,V}$, as well as a suitable consequence of the following density result, which may be of some independent interest. 
\begin{prop}
\label{prop_density_form}
Let $(M,g)$ be a conformally transversally anisotropic manifold of dimension $n\ge 3$, and let $A\in C^{1,1}(M, T^*M)$ be a $1$-form. If 
\begin{equation}
\label{eq_2_1}
\int_M\langle A, d(u_1u_2u_3)\rangle_g u_4 dV_g=0,
\end{equation}
for all harmonic functions $u_j\in C^\infty(M)$, $j=1,2,3,4$,  then $A\equiv 0$. 
\end{prop}

The starting point in the proof of Proposition \ref{prop_density_form} consists of showing that the boundary traces 
of the $1$-form $A$, as well as of its normal derivative, vanish, as a consequence of the integral identity \eqref{eq_2_1}. 
This allows us to extend $A$ by zero to $\R\times M_0\setminus M$, while preserving its regularity. The proof of Proposition \ref{prop_density_form} then follows the strategy of the proof of Proposition \ref{prop_density_potential} established in  \cite{LLLS}. Specifically, we construct harmonic functions to be used in  \eqref{eq_2_1}, based on suitable Gaussian beams quasimodes associated to two non-tangential intersecting geodesics on the transversal manifold $M_0$. Using the freedom of working with four harmonic functions, we construct a pair of harmonic functions based on a Gaussian beam quasimode $v$ and its complex conjugate $\overline{v}$, concentrated near one geodesic, and another pair of harmonic functions based on a Gaussian beam quasimode $w$ and its complex conjugate $\overline{w}$, concentrated near the other geodesic. The product $d(v \overline{v} w)\overline{w}$ is supported near the finitely many points of intersections of these geodesics, and the product does not have high oscillations. This makes it possible to conclude that $A=0$, using both non-stationary as well as stationary phase arguments (the Laplace method). 

\begin{rem}
Our regularity assumption on $A$  in Proposition \ref{prop_density_form} is motivated by the fact that the continuity of the zero extension of $A$ to $\R\times M_0\setminus M$ is needed for a rough stationary phase argument and the Lipschitz continuity of the gradient of the zero extension of $A$ is needed for a non-stationary phase argument in the proof of Proposition \ref{prop_density_form}.  
\end{rem}

Returning to the proof of Theorem \ref{thm_main}, let us mention that due to the assumptions ($A_{ii}$)  and ($V_{ii}$), ($V_{iii}$), only the linearizations of the Dirichlet--to--Neumann map of order $\geq 3$ become useful when recovering the nonlinear potentials $A(x,z)$ and $V(x,z)$. Considering the $m$th order linearization, $m\ge 3$, leads to the following integral identity, 
\begin{equation}
\label{eq_non_linear_m_integral}
\int_M \big((m+1)i \langle A, d(u_1\cdots u_m)\rangle_g u_{m+1}-\big(mi d^*(A) +V\big) u_1\cdots u_{m+1} \big)dV_g=0,
\end{equation}
where $A=A_{m-1}^{(1)}-A_{m-1}^{(2)}$ and $V=V_{m}^{(1)}-V_m^{(2)}$, which is valid for any $u_l\in C^{2,\alpha}(M)$ harmonic, $l=1,\dots, m+1$. Setting $u_1=\dots=u_{m-3}=1$ in \eqref{eq_non_linear_m_integral} gives the identity 
\begin{equation}
\label{eq_non_linear_m_integral_2}
\begin{aligned}
(m+1)i \int_M  \langle A, d(u_{m-2}  u_{m-1}u_{m})&\rangle_g  u_{m+1} dV_g\\
&= \int_M (mi d^*(A) +V) u_{m-2}  u_{m-1}u_{m} u_{m+1} \big)dV_g. 
\end{aligned}
\end{equation}
To proceed, we first show that  \eqref{eq_non_linear_m_integral_2} implies that $A|_{\p M}=0$ and $\p_\nu A|_{\p M}=0$, and then use a consequence of Proposition \ref{prop_density_form} to obtain that $A\equiv 0$, see Corollary \ref{cor_density_form} below. To recover $V$, we substitute $A=0$ in \eqref{eq_non_linear_m_integral_2}, and rely on Proposition \ref{prop_density_potential}. 

\begin{rem}
The assumptions $(A_i)$, $(A_{ii})$, $(V_i)$, $(V_{ii})$, $(V_{iii})$ in Theorem \ref{thm_main} are made precisely so that the higher order linearizations of the Dirichlet--to--Neumann map $\Lambda_{A,V}$ lead to the integral identities \eqref{eq_non_linear_m_integral}  involving at least four harmonic functions, and the freedom of working with four harmonic functions allows one to solve the inverse boundary problem without any assumption on the transversal manifold, cf. \cite{LLLS}.  
\end{rem}

Let us point out that inverse boundary problems for the nonlinear magnetic Schr\"odinger equation in the Euclidean space,  both in the case of full and partial data, have been studied recently in \cite{Lai_Ting}. The density of certain products of gradients of harmonic functions in the Euclidean space has been recently established in \cite{Carstea_Feizmohammadi}, when solving an inverse boundary problem for certain anisotropic quasilinear elliptic equations. 

Finally, let us remark that inverse boundary problems for nonlinear elliptic PDE have been studied extensively in the literature. We refer to \cite{Feizmohammadi_Oksanen}, \cite{LLLS},  \cite{LLLS_partial}, \cite{CNV_2019}, \cite{Carstea_Feizmohammadi}, \cite{Hervas_Sun},  \cite{IsaNach_1995}
\cite{IsaSyl_94},  \cite{Kang_Nak_02}, \cite{Sun_96}, \cite{Sun_2004}, \cite {Sun_2010}, \cite{Sun_Uhlm_97}, \cite{Krup_Uhlmann_non_linear_1}, \cite{Krup_Uhlmann_2}, \cite{Lai_Ting}, and the reference given there. 

The paper is organized as follows. In Section \ref{sec_harmonic_functions} we recall the construction of harmonic functions on conformally transversally anisotropic manifold based on Gaussian beams  quasimodes constructed on $\R\times M_0$ and localized near a non-tangential geodesics on the transversal manifold $M_0$. For convenience of the reader, in Section \ref{sec_simplified_setting}, we provide a proof of  Proposition \ref{prop_density_form} in a simplified setting. Section \ref{sec_general_setting} is devoted to the proof of Proposition \ref{prop_density_form} in the general case. The proof of Theorem \ref{thm_main} occupies Section \ref{sec_proof_thm_main}. Appendix \ref{sec_rough_stationary_phase} discusses a standard rough version of stationary phase needed in the proof of Proposition \ref{prop_density_form}.  In Appendix \ref{sec_well-posedness}, we show the well-posedness of the Dirichlet problem for the nonlinear magnetic Schr\"odinger equation, in the case of small boundary data.  The determination of the first order boundary traces 
of a scalar function and a $1$--form, via suitable orthogonality relations involving harmonic functions on the manifold $M$, is presented in Appendix \ref{app_boundary_determination}.  Finally, Appendix \ref{app_geodesics} discusses some basic properties of geodesics which are used in the body of the paper.

\section{Gaussian beams quasimodes and construction of harmonic functions}

\label{sec_harmonic_functions}

Let $(M,g)$ be a conformally transversally anisotropic manifold so that $(M,g)\subset\subset (\R\times M_0, c(e\oplus g_0))$. Let us write $x=(x_1, x')$ for local coordinates in $\R\times M_0$. Note that $\phi(x)=\pm \alpha x_1$, $\alpha>0$, is a limiting Carleman weight for $-h^2\Delta_{g}$, see \cite{DKSaloU_2009}. 

Letting $\tilde g=e\oplus g_0$, we have
\begin{equation}
\label{eq_3_1}
c^{\frac{n+2}{4}} \circ (-\Delta_g) \circ c^{-\frac{(n-2)}{4}} = -\Delta_{\tilde g}+ q, 
\end{equation}
where 
\[
 q=  - c^{\frac{n+2}{4}}\Delta_g(c^{-\frac{(n-2)}{4}}),
\]
see \cite{DKurylevLS_2016}.  Here $ q\in C^\infty(\R\times M_0; \R)$.  It follows from \eqref{eq_3_1} that in order to construct harmonic functions on $(M,g)$ based on Gaussian beams quasimodes, we shall need to have Gaussian beams quasimodes for the Schr\"odinger operator $-\Delta_{\tilde g}+q$,  conjugated by an exponential weight corresponding to the limiting Carleman weight $\phi$.  Our quasimodes will be constructed on the manifold $\R\times M_0$ and will be localized to non-tangential  geodesics on the transversal manifold $M_0$.  A unit speed geodesic $\gamma:[-S_1,S_2]\to M_0$, $0<S_1,S_2<\infty$, is called non-tangential if $\gamma(-S_1),\gamma(S_2)\in \p M_0$,  $\dot{\gamma}(-S_1), \dot{\gamma}(S_2)$ are non-tangential vectors to $\p M_0$ and $\gamma(t)\in M_0^{\text{int}}$ for all $-S_1<t<S_2$, see \cite{DKurylevLS_2016}.  As in  \cite{LLLS}, it will be convenient to normalize our quasimodes in $L^4(M_0)$, as later we shall have to deal with products of four such quasimodes. We shall need the following essentially well known result, see \cite[Section 4.1]{Feizmohammadi_Oksanen}, see also  \cite{DKurylevLS_2016}, \cite{LLLS}. 

\begin{prop}
\label{prop_Gaussian_beams}
Let $\alpha>0$, and let $\tau=s+i\lambda$,  $s\ge 1$, with $\lambda\in \R$ being fixed. Then for any $k\in \N$, $R\ge 1$,  there exist $N\in \N$ and families  of Gaussian beam quasimodes $v_1(\cdot; s), v_2(\cdot; s)\in C^\infty(\R\times M_0)$ such that 
\begin{equation}
\label{eq_prop_gaussian_1}
\begin{aligned}
 & \| e^{-\alpha \tau x_1}(- \Delta_{\tilde g} +q) e^{\alpha \tau x_1}v_1(\cdot; s)\|_{H^{k}((I\times M_0)^{\emph{\text{int}}})}=\mathcal{O}(s^{-R}),\\
&\| e^{\alpha \tau x_1}(- \Delta_{\tilde g} +q) e^{-\alpha \tau x_1}v_2(\cdot; s)\|_{H^{k}((I\times M_0)^{\emph{\text{int}}})}=\mathcal{O}(s^{-R}),
\end{aligned}
\end{equation}
and
\begin{equation}
\label{eq_prop_gaussian_2}
\begin{aligned}
\|v_j(\cdot;s)\|_{L^4(I\times M_0)}=\mathcal{O}(1),\quad \|v_j(\cdot;s)\|_{L^\infty(I\times M_0)}=\mathcal{O}(1)s^{\frac{n-2}{8}},\quad j=1,2,
\end{aligned}
\end{equation}
as $s\to \infty$. Here $I\subset \R$ is an arbitrary bounded interval.  The local structure of the quasimodes is as follows. Let 
$p\in \gamma([-S_1,S_2])$ and let $t_1<\dots<t_P$ be the times in $[-S_1,S_2]$ when $\gamma(t_l)=p$. In a sufficiently small neighborhood $U$ of  $p$, the quasimode $v_j$ is a finite sum,
\[
v_j|_U=v_j^{(1)}+\dots + v_j^{(P)}.
\] 
Each $v_j^{(l)}$ has the form
\[
v_1^{(l)}=s^{\frac{n-2}{8}}e^{i\alpha \tau \varphi^{(l)}}a^{(l)}, \quad v_2^{(l)}=s^{\frac{n-2}{8}}e^{i\alpha \tau \varphi^{(l)}}b^{(l)}\quad  l=1, \dots, P, 
\]
where  $\varphi=\varphi^{(l)}\in C^\infty(\overline{U};\C)$ satisfies for $t$ close to $t_l$, 
\[
\varphi(\gamma(t))=t, \quad \nabla \varphi(\gamma(t))=\dot{\gamma}(t), \quad \emph{\text{Im}}\,(\nabla^2\varphi(\gamma(t)))\ge 0, \quad  \emph{\text{Im}}\,(\nabla^2\varphi)|_{\dot{\gamma}(t)^\perp}>0, 
\] 
and $a^{(l)}, b^{(l)}\in C^\infty(\R\times \overline{U})$ are of the form, 
\[
a^{(l)}(x_1, t,y;s)=\bigg(\sum_{j=0}^N\tau^{-j}a_j^{(l)} \bigg) \chi\bigg(\frac{y}{\delta'}\bigg),\quad b^{(l)}(x_1, t,y;s)=\bigg(\sum_{j=0}^N\tau^{-j}b_j^{(l)} \bigg) \chi\bigg(\frac{y}{\delta'}\bigg),
\]
where $a^{(l)}_0=b^{(l)}_0$ is independent of $x_1$ and the potential $q$, 
\[
a^{(l)}_0(t,y)=a_{00}^{(l)}(t)+\mathcal{O}(|y|), \quad a_{00}^{(l)}(t)\ne 0, \quad  \forall t, 
\]
\begin{align*}
  a_1^{(l)}(x_1,t,y)=a_{10}^{(l)}(x_1, t)+\mathcal{O}(|y|),\quad b_1^{(l)}(x_1,t,y)=b_{10}^{(l)}(x_1, t)+\mathcal{O}(|y|).
\end{align*}
Here $a_{10}^{(l)}(x_1, t)=e^{f^{(l)}(t)}\tilde a_{10}^{(l)}(x_1, t)$, $b_{10}^{(l)}(x_1, t)=e^{f^{(l)}(t)}\tilde b_{10}^{(l)}(x_1, t)$, where $f^{(l)}$ is independent of the potential $q$ and $\tilde a_{10}^{(l)}$, $\tilde b_{10}^{(l)}$  satisfy the equations, 
\begin{align*}
&(\p_{x_1}+i \p_t)\tilde a_{10}^{(l)}=\frac{1}{\alpha}\bigg(-\frac{1}{2}e^{-f^{(l)}}(\Delta_{\tilde g} a_0^{(l)})|_{y=0}+C^{(l)}_0q(x_1,t,0)\bigg),\\
&(\p_{x_1}-i \p_t)\tilde b_{10}^{(l)}=\frac{1}{\alpha}\bigg(\frac{1}{2}e^{-f^{(l)}}(\Delta_{\tilde g} a_0^{(l)})|_{y=0}-C_0^{(l)}q(x_1,t,0)\bigg),
\end{align*}
where $C_0^{(l)}\ne 0$ is a constant, independent of the potential $q$. 
Here $(t,y)$ are the Fermi coordinates for $\gamma$ for $t$ close to $t_l$, $\chi\in C^\infty_0(\R^{n-2})$ is such that $0\le \chi\le 1$,   $\chi=1$ for $|y|\le 1/4$ and $\chi=0$ for $|y|\ge 1/2$, and $\delta'>0$ is a fixed number that can be taken arbitrarily small. 

\end{prop}

\begin{rem}
\label{rem_Gaussian_beams}
In the special case when the conformal factor $c=1$, we have $q=0$, $g=\tilde g$, and 
\[
e^{\mp\alpha\tau x_1}\circ(-\Delta_{g}) \circ e^{\pm\alpha\tau x_1}=-\Delta_{g}\mp 2\alpha \p_{x_1}-(\alpha \tau)^2.
\]  
Thus, we can take the Gaussian beams quasimodes  in \eqref{eq_prop_gaussian_1}  $v_1=v_2$ independent of $x_1$.
\end{rem}

Next we shall construct harmonic functions on $(M,g)$ based on the Gaussian beams quasimodes of Proposition \ref{prop_Gaussian_beams}. To that end, we shall use the approach of \cite{DKSaloU_2009}, based on Carleman estimates with limiting Carleman weights.  The construction is standard, see  \cite{DKurylevLS_2016},  \cite{LLLS}, and is presented here for the convenience of the reader only.

Assume, as we may,  that $(M,g)$ is embedded in a compact smooth manifold $(N,g)$ without boundary of the same dimension.  Our starting point is the following Carleman estimates for the Schr\"odinger operator, which is established in  \cite[Lemma 4.3]{DKSaloU_2009}. 

\begin{prop}
Let $q\in C^\infty(M)$.  Then given any  $t\in \R$,  we have for  all $h>0$ small enough and all $u\in C_0^\infty(M^0)$,  
\begin{equation}
\label{eq_Car_for_schr}
h\|u\|_{H^{t}_{\emph{\text{scl}}}(N)}\le C\|e^{\frac{\phi }{h}}(-h^2\Delta+h^2q)e^{-\frac{\phi}{h}}  u\|_{H^{t}_{\emph{\text{scl}}}(N)}, \quad C>0.
\end{equation}
\end{prop} 
Here $H^t(N)$, $t\in\R$, is the standard  Sobolev space, equipped with the natural semiclassical norm,
\[
\|u\|_{H^t_{\text{scl}}(N)}=\|(1-h^2\Delta_g)^{\frac{t}{2}} u\|_{L^2(N)}. 
\]
Using a standard argument, see \cite{DKSaloU_2009}, we convert the Carleman estimate \eqref{eq_Car_for_schr} into the following solvability result. 

\begin{prop}
\label{prop_solvability}
Let  $t\in \R$.  If  $h>0$ is small enough, then for any $v\in H^{t}(M^{\emph{\text{int}}})$, there is a solution $u\in H^t(M^{\emph{\text{int}}})$ of the equation 
\[
e^{\frac{\phi}{h}}(-h^2\Delta +h^2q)e^{-\frac{\phi}{h}}  u=v \quad \text{in}\quad M^{\emph{\text{int}}},
\]
which satisfies 
\[
\|u\|_{H^t_{\emph{\text{scl}}}(M^{\emph{\text{int}}})}\le \frac{C}{h} \|v\|_{H^{t}_{\emph{\text{scl}}}(M^{\emph{\text{int}}})}.
\]
\end{prop} 
Here 
\[
H^t(M^{\emph{\text{int}}})=\{V|_{M^{\emph{\text{int}}}}: V\in H^t(N)\}, \quad t\in \R,
\]
with the norm
\[
\|v\|_{H^t_{\emph{\text{scl}}}(M^0)}=\inf_{V\in H^t_{\emph{\text{scl}}}(N),v=V|_{M^{\text{int}}}}\|V\|_{H^t_{\emph{\text{scl}}}(N)}.
\]

Let  $\alpha>0$, and let 
\[
\tau=s+i\lambda, \quad 1\le s=\frac{1}{h}, \quad \lambda\in \R, \quad \lambda\quad  \text{fixed}. 
\]
In view of \eqref{eq_3_1}, to construct suitable harmonic functions on $(M,g)$, we shall find complex geometric optics solution to the equation 
\begin{equation}
\label{eq_CGO_6}
(-\Delta_{\tilde g}+q) \tilde u=0\quad \text{in}\quad M^{\text{int}},
\end{equation}
having the form
\[
\tilde u_1=e^{\alpha \tau x_1}(v_1+r_1), \quad \tilde u_2=e^{-\alpha \tau x_1}(v_2+r_2)
\]
where $v_1$, $v_2$ are the Gaussian beam quasimodes given in Proposition \ref{prop_Gaussian_beams}, and $r_1$, $r_2$ are the remainder terms.   
Thus,  $\tilde u_1$  is a solution of \eqref{eq_CGO_6}
 provided that 
 \begin{equation}
 \label{eq_cgo_conjug}
 e^{-  \frac{\alpha x_1}{h}} (-h^2\Delta_{\tilde g}+h^2 q ) e^{\frac{\alpha x_1}{h}}(e^{i\alpha \lambda x_1}r_1) =- e^{i \alpha \lambda x_1} e^{-\alpha \tau x_1}(-h^2\Delta_{\tilde g}+h^2 q ) e^{\alpha \tau x_1}v_1.
 \end{equation}
For any $k\in \N$, $R\ge 1$, arbitrarily large,  Proposition \ref{prop_solvability} and Proposition \ref{prop_Gaussian_beams} imply that there is $r\in H^k(M^\text{int})$ such that 
\[
\|r_1\|_{H^k_{\text{scl}}(M^{\text{int}})}\le \mathcal{O}(h^{-1})\|e^{-\alpha \tau x_1} (-h^2\Delta_{\tilde g}+h^2 q ) e^{\alpha \tau x_1}v_1\|_{H^k_{\text{scl}}(M^{\text{int}})}=\mathcal{O}(h^{R-1}),
\]
and therefore, for any $K$, there is $R$ large enough so that
\[
\|r_1\|_{H^k(M^{\text{int}})}\le h^{-k}\|r_1\|_{H^k_{\text{scl}}(M^{\text{int}})}=\mathcal{O}(h^K).
\]
Similarly, one can construct $r_2$. This together with \eqref{eq_3_1} gives the following result concerning the construction of harmonic functions on $(M,g)$ based on Gaussian beams quasimodes. 

\begin{prop}
\label{prop_CGO_general_mnfld}
Let $\alpha>0$, and let $\tau=s+i\lambda$, $s=\frac{1}{h}$, with $\lambda\in \R$ being fixed.  For all $k$, $K$, and  $h>0$ small enough, there are  $u_1, u_2\in H^k(M^{\emph{\text{int}}})$ of $-\Delta_g u_j=0$ in $M^{\emph{\text{int}}}$ having the form
\[
u_1=e^{\alpha \tau x_1} c^{-\frac{(n-2)}{4}} (v_1+r_1), \quad u_2=e^{-\alpha \tau x_1} c^{-\frac{(n-2)}{4}} (v_2+r_2),
\]
where $v_1=v_1(\cdot; s), v_2=v_2(\cdot; s)\in C^\infty(\R\times M_0)$ are the Gaussian beam quasimodes given in Proposition \ref{prop_Gaussian_beams}, and $r_1, r_2\in H^k(M^{\emph{\text{int}}})$ are such that $\|r_j\|_{H^k(M^{\emph{\text{int}}})}=\mathcal{O}(h^K)$ as $h\to 0$.
\end{prop}

\begin{rem}
\label{rem_CGO_general_mnfld}
Taking $k>n/2+3$ and using the Sobolev embedding $H^k(M^{\emph{\text{int}}})\subset C^3(M)$, we see that $u_j\in C^3(M)$ with 
\[
\|r_j\|_{C^3(M)}=\mathcal{O}(h^K), 
\]
as $h\to 0$, $j=1,2$. 
\end{rem}

\section{Proof of Proposition \ref{prop_density_form} in a simplified setting}

\label{sec_simplified_setting}

The proof of Proposition \ref{prop_density_form} will follow along the lines of the proof of Proposition 5.1 in \cite{LLLS}. Before we prove Proposition \ref{prop_density_form} in the general case, let us explain the main ideas in a simplified setting. 

Let us assume that each point $p\in M_0^{\text{int}}$ is the unique intersection point of two distinct non-tangential 
non-self-intersecting geodesics $\gamma$ and $\eta$. Assume furthermore that the conformal factor $c=1$. As we shall see below, in this simplified setting the continuity of $A$ suffices, and therefore to extend $A$ by $0$ to the continuous form on $\R\times M_0\setminus M$, we only need to show $A|_{\p M}=0$.  This follows by taking $u_2=u_3=1$ in \eqref{eq_2_1} and applying Proposition \ref{prop_boundary_A}. 

In view of Proposition \ref{prop_density_Holder}, we see that \eqref{eq_2_1} also holds for all harmonic functions $u_j\in C^{2,\alpha}(M)$, $0<\alpha<1$, $j=1,\dots, 4$. 

Let $s=\frac{1}{h}$, and let $\lambda\in \R$ be fixed. Our choice of the harmonic functions below will be similar to \cite{LLLS}. Specifically, using Proposition \ref{prop_CGO_general_mnfld} and Remark \ref{rem_CGO_general_mnfld}, we see that there exist harmonic functions $u_j\in C^3(M)$, $j=1,\dots, 4$, on $(M,g)$ of the form
\begin{equation}
\label{eq_4_1}
\begin{aligned}
&u_1=e^{-(s+i\lambda)x_1}(v+r_1), \quad u_2=\overline{e^{(s+i\lambda)x_1}(v+r_2)},\\
&u_3=e^{-sx_1}(w+r_3), \quad u_4=\overline{e^{sx_1}(w+r_4)},
\end{aligned}
\end{equation}
where 
\begin{equation}
\label{eq_4_2}
\|r_j\|_{C^1(M)}=\mathcal{O}(s^{-K}),
\end{equation} 
as $s\to \infty$, $K\gg 1$, and $v=v(\cdot;s), w=w(\cdot;s)\in C^\infty(M_0)$ are Gaussian beams quasimodes concentrating near the geodesics $\eta$ and $\gamma$, respectively, constructed in Proposition \ref{prop_Gaussian_beams}, see also Remark \ref{rem_Gaussian_beams}. We have 
\begin{equation}
\label{eq_4_3}
v(x';s)=s^{\frac{n-2}{8}}e^{i(s+i\lambda)\varphi(x')}a(x';s), \quad w(x';s)=s^{\frac{n-2}{8}}e^{i s\psi(x')}b(x';s),
\end{equation}
where 
\begin{equation}
\label{eq_4_4}
\begin{aligned}
\varphi(\eta(t))=t, \quad \nabla \varphi(\eta(t))=\dot{\eta}(t), \quad \text{Im}\,(\nabla^2\varphi(\eta(t)))\ge 0, \quad  \text{Im}\,(\nabla^2\varphi)|_{\dot{\eta}(t)^\perp}>0, \\
\psi(\gamma(\tau))=\tau, \quad \nabla \psi(\gamma(\tau))=\dot{\gamma}(\tau), \quad \text{Im}\,(\nabla^2\psi(\gamma(\tau)))\ge 0, \quad  \text{Im}\,(\nabla^2\psi)|_{\dot{\gamma}(\tau)^\perp}>0,
\end{aligned}
\end{equation}
and 
\begin{equation}
\label{eq_4_5}
a(t,y;s)=\bigg(\sum_{j=0}^N\tau^{-j}a_j \bigg) \chi\bigg(\frac{y}{\delta'}\bigg),\quad b(\tau, z;s)=\bigg(\sum_{j=0}^N\tau^{-j}b_j \bigg) \chi\bigg(\frac{z}{\delta'}\bigg),
\end{equation}
where
\begin{equation}
\label{eq_4_6}
\begin{aligned}
&a_0(t,y)=a_{00}^{(l)}(t)+\mathcal{O}(|y|), \quad a_{00}(t)\ne 0, \quad  \forall t, \\
&b_0(\tau,z)=a_{00}(\tau)+\mathcal{O}(|z|), \quad b_{00}(\tau)\ne 0, \quad  \forall \tau.
\end{aligned}
\end{equation}
Here $(t,y)$ and $(\tau,z)$ are the Fermi coordinates for the geodesics $\eta$ and $\gamma$,  $\chi\in C^\infty_0(\R^{n-2})$ is such that $0\le \chi\le 1$,   $\chi=1$ for $|y|\le 1/4$ and $\chi=0$ for $|y|\ge 1/2$, and $\delta'>0$ is a fixed number that can be taken arbitrarily small. We also have 
\begin{equation}
\label{eq_4_7}
\|v\|_{L^4(M_0)}=\|w\|_{L^4(M_0)}=\mathcal{O}(1), \quad \|v\|_{L^\infty(M_0)}=\|w\|_{L^\infty(M_0)}=\mathcal{O}(s^{\frac{n-2}{8}}),
\end{equation}
as $s\to \infty$.  Similarly, we find that 
\begin{equation}
\label{eq_4_8}
\begin{aligned}
&\|s^{\frac{n-2}{8}}e^{i(s+i\lambda)\varphi}\nabla a \|_{L^4(M_0)}=\|s^{\frac{n-2}{8}}e^{is\psi}\nabla b\|_{L^4(M_0)}=\mathcal{O}(1),\\
&\|\nabla v\|_{L^4(M_0)}=\mathcal{O}(s), \quad \|\nabla w\|_{L^4(M_0)}=\mathcal{O}(s),\\ 
&\|\nabla v\|_{L^\infty(M_0)}=\mathcal{O}(s^{\frac{n+6}{8}}), \quad \|\nabla w\|_{L^\infty(M_0)}=\mathcal{O}(s^{\frac{n+6}{8}}),
\end{aligned} 
\end{equation}
as $s\to \infty$.

Now it follows from \eqref{eq_4_1} that 
\[
(u_1u_2u_3)(x)= e^{-2i\lambda x_1-s x_1} (|v(x')|^2w(x')+R(x)), 
\]
where
\[
R=|v|^2r_3+(w+r_3)(v\overline{r_2}+\overline{v}r_1+r_1\overline{r_2}).
\]
Using \eqref{eq_4_2}, \eqref{eq_4_7}, and \eqref{eq_4_8}, we see that 
\begin{equation}
\label{eq_rem_R}
\|R\|_{C^1(M)}=\mathcal{O}(s^{-L}),
\end{equation}
where $L$ is large depending on $K$.  Hence, we have 
\[
\p_{x_1}(u_1u_2u_3)= e^{-2i\lambda x_1-s x_1} [(-2i\lambda -s )(|v|^2w+R)+\p_{x_1}R],
\]
and therefore, using \eqref{eq_rem_R}, \eqref{eq_4_2},  and \eqref{eq_4_7}, we get 
\begin{equation}
\label{eq_4_13}
\p_{x_1}(u_1u_2u_3)u_4=-s e^{-2i\lambda x_1} |v|^2|w|^2+\mathcal{O}_{L^1(M)}(1),
\end{equation}
as $s\to \infty$. We also get 
\[
\p_{x_k}(u_1u_2u_3)= e^{-2i\lambda x_1-s x_1} (\p_{x_k}(|v|^2w)+\p_{x_k}(R)),
\]
for $k=2,\dots, n$, and therefore, \eqref{eq_rem_R}, \eqref{eq_4_2}, \eqref{eq_4_7}, and \eqref{eq_4_8}, 
\begin{equation}
\label{eq_4_13_1}
\p_{x_k}(u_1u_2u_3)u_4= e^{-2i\lambda x_1} \p_{x_k}(|v|^2w)\overline{w}+\mathcal{O}_{L^1(M)}(1),
\end{equation}
as $s\to \infty$. Writing $A=(A_1, A')$, and using  \eqref{eq_4_13}, \eqref{eq_4_13_1},  we conclude that 
\begin{equation}
\label{eq_4_14}
\langle A,d(u_1u_2u_3) \rangle_{g}u_4=e^{-2i\lambda x_1}(-s A_1|v|^2|w|^2+ \langle A', d_{x'}(|v|^2w)\overline{w}\rangle_{g_0})+\mathcal{O}_{L^1(M)}(1),
\end{equation}
as $s\to \infty$. It follows from \eqref{eq_2_1} with the help of \eqref{eq_4_14} that 
\begin{equation}
\label{eq_4_15}
\int_M e^{-2i\lambda x_1} (-s A_1|v|^2|w|^2+ \langle A', d_{x'}(|v|^2w)\overline{w}\rangle_{g_0})dV_g=\mathcal{O}(1),
\end{equation}
as $s\to \infty$. 

Extending $A$ by zero to $\R\times M_0\setminus M$, and denoting the extension again by $A$, we see that $A\in C(\R\times M_0, T^*(\R\times M_0))$ as $A|_{\p M}=0$.  Denoting the partial Fourier transform of $A$ in the $x_1$ variable by $\hat{A}(\lambda, x')$, we get from \eqref{eq_4_15} that 
\begin{equation}
\label{eq_4_16}
\int_{M_0} (-s \hat A_1 (2\lambda, \cdot)|v|^2|w|^2+ \langle \hat A'(2\lambda,\cdot ), d_{x'}(|v|^2w)\overline{w}\rangle_{g_0})dV_{g_0}=\mathcal{O}(1),
\end{equation}
as $s\to \infty$. Since $v$ and $w$ can be chosen to be supported in arbitrarily small but fixed neighborhoods of $\eta$ and $\gamma$, respectively, and since $\eta$ and $\gamma$ only intersect at $p$, the products $|v|^2|w|^2$ and $d_{x'}(|v|^2w)\overline{w}$ concentrate in a small neighborhood $U$ of $p$.  Using \eqref{eq_4_3} and \eqref{eq_4_5}, we see that in $U$, 
\begin{equation}
\label{eq_4_17}
\begin{aligned}
|v|^2|w|^2=s^{\frac{n-2}{2}}e^{-2s(\text{Im}\, \varphi+ \text{Im}\, \psi)}e^{-2\lambda\text{Re}\,\varphi} (|a_0|^2|b_0|^2+\mathcal{O}_{L^\infty(M_0)}(1/s))\\
= s^{\frac{n-2}{2}}e^{-2s(\text{Im}\, \varphi+ \text{Im}\, \psi)}e^{-2\lambda\text{Re}\,\varphi} |a_0|^2|b_0|^2+\mathcal{O}_{L^1(M_0)}(1/s),
\end{aligned}
\end{equation}
and 
\begin{equation}
\label{eq_4_18}
\begin{aligned}
d_{x'}(|v|^2w)\overline{w}=&s^{\frac{n-2}{2}}e^{-2s(\text{Im}\, \varphi+ \text{Im}\, \psi)}e^{-2\lambda\text{Re}\,\varphi} [(is(2 id\text{Im}\, \varphi+  d\psi)\\
 & (|a_0|^2|b_0|^2+\mathcal{O}_{L^\infty(M_0)}(1/s))
-2\lambda (d\text{Re}\,\varphi) |a|^2|b|^2+  d_{x'}(|a|^2b)\overline{b}]\\
=&s^{\frac{n-2}{2}}e^{-2s(\text{Im}\, \varphi+ \text{Im}\, \psi)}e^{-2\lambda\text{Re}\,\varphi} is(2 id\text{Im}\, \varphi+  d\psi) |a_0|^2|b_0|^2+\mathcal{O}_{L^1(M_0)}(1),
\end{aligned}
\end{equation}
as $s\to \infty$. Substituting \eqref{eq_4_17} and \eqref{eq_4_18} into \eqref{eq_4_16}, and dividing by $s^{\frac{1}{2}}$, we obtain that 
\begin{equation}
\label{eq_4_19}
\begin{aligned}
s^{\frac{n-1}{2}}\int_{U} (- \hat A_1 (2\lambda, \cdot)+ i\langle \hat A'(2\lambda,\cdot ), 2i d\text{Im}\,\varphi+d\psi\rangle_{g_0})e^{-2\lambda\text{Re}\, \varphi}|a_0|^2|b_0|^2e^{-s\Psi}dV_{g_0}\\
=\mathcal{O}(s^{-\frac{1}{2}}),
\end{aligned}
\end{equation}
as $s\to \infty$, where
\[
\Psi=2(\text{Im}\, \varphi+ \text{Im}\, \psi).
\] 
It follows from \eqref{eq_4_4} that 
\[
\Psi(p)=0, \quad d\Psi(p)=0, \quad \nabla^2\Psi(p)>0,
\]
where the later inequality is a consequence of the fact that the Hessians of $\text{Im}\, \varphi$ and $\text{Im}\, \psi$ at $p$ are positive definite in the directions orthogonal to $\eta$ and $\gamma$, respectively. 

Let us now denote by $z=(z_1,\dots, z_{n-1})$  geodesic normal coordinates in $(M_0,g_0)$ with the origin at $p$. Then 
\begin{equation}
\label{eq_4_20}
g_0(z)=1+\mathcal{O}(|z|^2),
\end{equation}
see \cite[Chapter 2, Section 8, p. 56]{Petersen_book}, and $dV_{g_0}=|g_0(z)|^{1/2}dz$. Passing to the limit as $s\to \infty$ in \eqref{eq_4_19} and using the rough version of stationary phase Lemma \ref{lem_stationary_phase}, as well as \eqref{eq_4_20}, we obtain that 
\[
(- \hat A_1 (2\lambda, p)+ i\hat A'(2\lambda,p )(\dot{\gamma}(t_0))e^{-2\lambda\text{Re}\,\varphi(p)}|a_{00}(p)|^2|b_{00}(p)|^2=0, 
\]
where $p=\gamma(t_0)$, 
for all $\lambda\in \R$. As $a_{00}(p)\ne 0$, $b_{00}(p)\ne 0$, and $\lambda$ is arbitrary,  we see that 
\[
- A_1 (x_1, p)+ i A'(x_1,p )(\dot{\gamma}(t_0))=0,
\]
which is equivalent to 
\begin{equation}
\label{eq_4_21}
(iA_1, A')(x_1,p)(1,\dot{\gamma}(t_0))=0.
\end{equation}
Here we may replace $\dot{\gamma}(t_0)$ by $-\dot{\gamma}(t_0)$. Thus, \eqref{eq_4_21} gives that $A_1(x_1,p)=0$, and since the point $(x_1,p)$ is an arbitrary point in $\R\times M_0$, we get $A_1\equiv 0$. Hence, we only need to recover the $1$-form $A'(x_1,\cdot)$ on $M_0$ knowing that  
\begin{equation}
\label{eq_4_21_2}
A'(x_1,p)(\dot{\gamma}(t_0))=0.
\end{equation}
To that end, we assume without loss of generality that $v_1=\dot{\gamma}(t_0)=(1,0, \dots, 0)\in \R^{n-1}$, and consider the small perturbations of $v_1$ given by 
\begin{equation}
\label{eq_4_22-new}
v_2=\frac{1}{\sqrt{1+\varepsilon^2}}(1,\varepsilon, 0, \dots, 0), \dots, v_{n-1}=\frac{1}{\sqrt{1+\varepsilon^2}}(1,0, \dots, 0,\varepsilon),
\end{equation}
for $\varepsilon>0$ small. The unit vectors $v_1, \dots, v_{n-1}$ are linearly independent, and thus, they span the tangent space $T_p M_0$. By Proposition \ref{prop_D_geodesics}, for $\varepsilon>0$ sufficiently small, the unit speed geodesic $\gamma_{p,v_j}$, $j=2,\dots, n-1$, through $(p,v_j)$ is non-tangential between boundary points, does not have self-intersections, and intersects $\eta$ at the point $p$ only.  Applying the discussion above with $\gamma=\gamma_{p,v_j}$, we obtain that $A'(x_1,p)(v_j)=0$, $j=2,\dots, n-1$. This together with \eqref{eq_4_21_2} gives that $A'(x_1,p)=0$. The proof of Proposition \ref{prop_density_form}  in the simplified case is complete.

\section{Proof of Proposition \ref{prop_density_form} in the general setting}

\label{sec_general_setting}

In the case of general transversal manifold $M_0$, the non-tangential geodesics $\gamma$ and $\eta$ might have self-intersections and may intersect more than in one point, which complicates the proof. To proceed we shall follow \cite{LLLS} and introduce additional parameters in the construction of harmonic functions. Furthermore, we shall implement the presence of the conformal factor $c$ which is assumed to be equal to $1$ in \cite{LLLS}.

Let us proceed to discuss the choice of two non-tangential geodesics to be used when constructing Gaussian beams quasimodes. When doing so let us first observe that arguing as in the proof of Theorem 1.2 of \cite{Salo_2017}, we may assume that $(M_0,g_0)$ has a strictly convex boundary. An application of 
\cite[Lemma 3.1]{Salo_2017} gives therefore that there exists a nullset $E$ in $(M_0, g_0)$ such that all points in $M_0\setminus E$ lie on some non-tangential geodesic joining boundary points. Fix a point $y_0\in M_0^{\text{int}}\setminus E$ and let $\gamma:[-S_1, S_2]\to M_0$, $0<S_1, S_2<\infty$, be a unit speed non-tangential geodesic such that $\gamma(0)=y_0$. Then by Proposition \ref{prop_geodesic}, moving the point $y_0$ along $\gamma$ a little and reparametrizing the geodesic, if necessary, there exists a small neighborhood $W\subset S_{y_0}M_0$ of $w_0=\dot{\gamma}(0)$ such that for every $w\in W$, $w \neq w_0$,  the unit speed geodesic $\eta:[-T_1,T_2]\to M_0$, $0<T_1,T_2<\infty$, such that $\eta(0)=y_0$ and $\dot{\eta}(0)=w$  is also non-tangential, and $\gamma$ and $\eta$ do not  intersect each other at the boundary of $M_0$. Notice that $\gamma$ and $\eta$ are distinct and are not reverses of each other. As we shall see below, the fact that $\gamma$ and $\eta$  do not  intersect each other at the boundary of $M_0$ allows us to avoid the use of stationary and non-stationary phase on the boundary of $M_0$.

By \cite{LLLS}, we know that $\gamma$ and $\eta$ can intersect only finitely many times.  Let us denote by $p_1, \dots, p_N\in M_0^{\text{int}}$ the distinct intersection points of $\gamma$ and $\eta$. For each $r$, $r=1,\dots, N$, let $t_1^{(r)}<\dots<t_{P_r}^{(r)}$ be the times in $[-T_1,T_2]$ when $\eta(t_j^r)=p_r$, and let $\tau_1^{(r)}<\dots<\tau_{Q_r}^{(r)}$ be the times in  $[-S_1,S_2]$ when $\gamma(\tau_j^{(r)})=p_r$. Let $U_r$ be a  small neighborhood of $p_r$, $r=1,\dots, N$. 

\subsection{Choosing harmonic functions}

First it follows from Proposition \ref{prop_density_Holder} that \eqref{eq_2_1} continues to hold for all harmonic functions $u_j\in C^{2,\alpha}(M)$, $0<\alpha<1$, $j=1,\dots, 4$.

Let $s\ge 1$ and let $L>0$, $\lambda, \mu\in \R$ be fixed. By Proposition \ref{prop_CGO_general_mnfld} and Remark \ref{rem_CGO_general_mnfld}, there are harmonic functions $u_j\in C^3(M)$ of the form, 
\begin{equation}
\label{eq_5_1}
\begin{aligned}
u_1=e^{(s+i\mu)x_1}c^{-\frac{(n-2)}{4}}(v_1+r_1), \quad u_2=\overline{e^{-(s+i\mu)x_1}c^{-\frac{(n-2)}{4}}(v_2+r_2)},\\
u_3=e^{-L(s+i\lambda)x_1}c^{-\frac{(n-2)}{4}}(w_1+r_3), \quad u_4=\overline{e^{L(s+i\lambda)x_1}c^{-\frac{(n-2)}{4}}(w_2+r_4)},
\end{aligned}
\end{equation}
where 
\begin{equation}
\label{eq_5_1_2}
\|r_j\|_{C^1(M)}=\mathcal{O}(s^{-K}),
\end{equation}
as $s\to \infty$, $K\gg 1$, and  $v_j\in C^\infty(\R\times M_0)$, $j=1,2$, and $w_j\in C^\infty(\R\times M_0)$, $j=1,2$, are the Gaussian beam quasimodes constructed in Proposition \ref{prop_Gaussian_beams} and associated to the non-tangential geodesics $\eta$ and $\gamma$, respectively, such that  
\begin{equation}
\label{eq_5_1_3}
\supp(v_j(\cdot;s))\subset \R\times \text{small neigh}(\eta)\quad \supp(w_j(\cdot;s))\subset \R\times \text{small neigh}(\gamma).
\end{equation}
Notice that here we follow  \cite{LLLS}, and the minor differences are as follows: in order to incorporate the presence of the conformal factor our Gaussian beams quasimodes are constructed on all of $\R\times M_0$ rather than on $M_0$ as in  \cite{LLLS}, and the parameters $\mu$ and $\lambda$ are real.

Let us now recall a local description of the quasimodes $v_j$, $w_j$ near the intersection points $p_r$ of $\gamma$ and $\eta$. In doing so, let us fix $p$ to be one of the intersection points $p_r$ and let us set $U=U_r$.  In the open set $U$, the quasimodes $v_j$ are of the form 
\begin{equation}
\label{eq_5_2}
v_j|_U=\sum_{k=1}^P v^{(k)}_j, \quad j=1,2, 
\end{equation}
where $t_1<\dots<t_P$ are the times in $[-T_1,T_2]$ when $\eta(t_k)=p$.  
Each $v_1^{(k)}$, $v_2^{(k)}$ in \eqref{eq_5_2} has the form
\begin{equation}
\label{eq_5_3}
v_1^{(k)}=s^{\frac{n-2}{8}}e^{i (s+i\mu) \varphi^{(k)}}a^{(k)}, \quad v_2^{(k)}=s^{\frac{n-2}{8}}e^{i (s+i\mu) \varphi^{(k)}}b^{(k)}, \quad k=1, \dots, P, 
\end{equation}
where $\varphi=\varphi^{(k)}\in C^\infty(\overline{U};\C)$, satisfying for $t$ close to $t_k$, 
\begin{equation}
\label{eq_5_4}
\varphi(\eta(t))=t, \quad \nabla \varphi(\eta(t))=\dot{\eta}(t), \quad \text{Im}\,(\nabla^2\varphi(\eta(t)))\ge 0, \quad  \text{Im}\,(\nabla^2\varphi)|_{\dot{\eta}(t)^\perp}>0, 
\end{equation}
and $a^{(k)}, b^{(k)}\in C^\infty(\R\times \overline{U})$ is of the form, 
\begin{equation}
\label{eq_5_5}
a^{(k)}(x_1, t,y;s)=\bigg(\sum_{j=0}^N\tau^{-j}a_j^{(k)} \bigg) \chi\bigg(\frac{y}{\delta'}\bigg),\quad b^{(k)}(x_1, t,y;s)=\bigg(\sum_{j=0}^N\tau^{-j}b_j^{(k)} \bigg) \chi\bigg(\frac{y}{\delta'}\bigg),
\end{equation}
where $a^{(k)}_0=b^{(k)}_0$ is independent of $x_1$, 
\begin{equation}
\label{eq_5_6}
a^{(k)}_0(t,y)=a_{00}^{(k)}(t)+\mathcal{O}(|y|), \quad a_{00}^{(k)}(t)\ne 0, \quad  \forall t. 
\end{equation}
Here $(t,y)$ are the Fermi coordinates for $\eta$ for $t$ close to $t_k$, $\chi\in C^\infty_0(\R^{n-2})$ is such that $0\le \chi\le 1$,   $\chi=1$ for $|y|\le 1/4$ and $\chi=0$ for $|y|\ge 1/2$, and $\delta'>0$ is a fixed number that can be taken arbitrarily small. 

Furthermore, in $U$, the quasimodes $w_1$, $w_2$ are finite sums,
\begin{equation}
\label{eq_5_7}
w_j|_U=\sum_{k=1}^Q w^{(k)}_j, \quad j=1,2, 
\end{equation}
where $\tau_1<\dots<\tau_Q$ are the times in $[-S_1,S_2]$ when $\gamma(\tau_k)=p$.  Each $w_1^{(k)}$, $w_2^{(k)}$ in \eqref{eq_5_7}  has the form 
\begin{equation}
\label{eq_5_8}
w_1^{(k)}=s^{\frac{n-2}{8}}e^{Li(s+i\lambda) \psi^{(k)}}c^{(k)}, \quad w_2^{(k)}=s^{\frac{n-2}{8}}e^{Li (s+i\lambda) \psi^{(k)}}d^{(k)}, \quad k=1, \dots, Q, 
\end{equation}
where each $\psi=\psi^{(k)}\in C^\infty(\overline{U};\C)$, satisfying for $\tau$ close to $\tau_k$, 
\begin{equation}
\label{eq_5_9}
\psi(\gamma(\tau))=\tau, \quad \nabla \psi(\gamma(\tau))=\dot{\gamma}(\tau), \quad \text{Im}\,(\nabla^2\psi(\gamma(\tau)))\ge 0, \quad  \text{Im}\,(\nabla^2\psi)|_{\dot{\gamma}(\tau)^\perp}>0, 
\end{equation}
and each $c^{(k)}, d^{(k)}\in C^\infty(\R\times \overline{U})$ is of the form, 
\begin{equation}
\label{eq_5_10}
c^{(k)}(x_1, \tau,z;s)=\bigg(\sum_{j=0}^N\tau^{-j}c_j^{(k)} \bigg) \chi\bigg(\frac{z}{\delta'}\bigg),\ d^{(k)}(x_1, \tau,z;s)=\bigg(\sum_{j=0}^N\tau^{-j}d_j^{(k)} \bigg) \chi\bigg(\frac{z}{\delta'}\bigg),
\end{equation}
where $c^{(k)}_0=d^{(k)}_0$ is independent of $x_1$, 
\begin{equation}
\label{eq_5_11}
c^{(k)}_0(\tau,z)=c_{00}^{(k)}(\tau)+\mathcal{O}(|z|), \quad c_{00}^{(k)}(\tau)\ne 0, \quad  \forall \tau. 
\end{equation}
Here $(\tau,z)$ are the Fermi coordinates for $\gamma$ for $t$ close to $t_k$. 

We also have
\begin{equation}
\label{eq_5_12}
\begin{aligned}
&\|v_j\|_{L^4(M)}=\mathcal{O}(1), \quad \|w_j\|_{L^4(M)}=\mathcal{O}(1), \\
 &\|v_j\|_{L^\infty(M)}=\mathcal{O}(s^{\frac{n-2}{8}}),  \quad \|w_j\|_{L^\infty(M)}=\mathcal{O}(s^{\frac{n-2}{8}}),\\
&\|\nabla v_j\|_{L^4(M)}=\mathcal{O}(s), \quad \|\nabla w_j\|_{L^4(M)}=\mathcal{O}(s),\\\
 &\|\nabla v_j\|_{L^\infty(M)}=\mathcal{O}(s^{\frac{n+6}{8}}), \quad \|\nabla w_j\|_{L^\infty(M)}=\mathcal{O}(s^{\frac{n+6}{8}}),
 \end{aligned}
\end{equation}
as $s\to \infty$, $j=1,2$. 

Now it follows from \eqref{eq_5_1} that 
\begin{equation}
\label{eq_5_13}
u_1u_2u_3=e^{(-Ls+2i\mu-Li\lambda)x_1} c^{-\frac{3(n-2)}{4}}(v_1\overline{v_2}w_1+\tilde R), 
\end{equation}
where 
\[
\tilde R=r_3v_1\overline{v_2}+(w_1+r_3)(v_1\overline{r_2}+\overline{v_2}r_1+r_1\overline{r_2}).
\]
Using \eqref{eq_5_1_2}, \eqref{eq_5_12},  we get 
\begin{equation}
\label{eq_5_13_2}
\|\tilde R\|_{C^1(M)}=\mathcal{O}(s^{-L}),
\end{equation}
where $L$ is large. 
Hence, we have 
\begin{align*}
\p_{x_1}(u_1u_2u_3)=&e^{(-Ls+2i\mu-Li\lambda)x_1} [(-Ls+2i\mu-Li\lambda)c^{-\frac{3(n-2)}{4}}(v_1\overline{v_2}w_1+\tilde R)\\
&+ \p_{x_1}(c^{-\frac{3(n-2)}{4}}) (v_1\overline{v_2}w_1+\tilde R)+ c^{-\frac{3(n-2)}{4}}(\p_{x_1}(v_1\overline{v_2}w_1)+\p_{x_1}\tilde R)],
\end{align*}
and therefore, in view of \eqref{eq_5_13_2},  \eqref{eq_5_1_2}, and \eqref{eq_5_12},  we get 
\begin{equation}
\label{eq_5_14}
\p_{x_1}(u_1u_2u_3)u_4=e^{2i(\mu-L\lambda)x_1}c^{-(n-2)} [-Ls v_1\overline{v_2}w_1\overline{w_2}
+\p_{x_1}(v_1\overline{v_2}w_1)\overline{w_2}]+ \mathcal{O}_{L^1(M)}(1),
\end{equation}
as $s\to \infty$. 
We also have from \eqref{eq_5_13} that 
\begin{align*}
\p_{x_k}(u_1u_2u_3)=&e^{(-Ls+2i\mu-Li\lambda)x_1} [ c^{-\frac{3(n-2)}{4}}(\p_{x_k}(v_1\overline{v_2}w_1)+\p_{x_k}\tilde R)\\
&+\p_{x_k}(c^{-\frac{3(n-2)}{4}})(v_1\overline{v_2}w_1+\tilde R)],
\end{align*}
for $k=2,\dots,n$, and  therefore, in view of \eqref{eq_5_13_2}, \eqref{eq_5_1_2}, and \eqref{eq_5_12}, we get 
\begin{equation}
\label{eq_5_15}
\p_{x_k}(u_1u_2u_3)u_4=e^{2i(\mu-L\lambda)x_1}c^{-(n-2)} \p_{x_k}(v_1\overline{v_2}w_1)\overline{w_2} + \mathcal{O}_{L^1(M)}(1),
\end{equation}
as $s\to \infty$. 

For future reference, we also note that 
\begin{equation}
\label{eq_5_15_2}
u_1u_2u_3u_4=e^{2i(\mu-L\lambda)x_1}c^{-(n-2)} (v_1\overline{v_2}w_1\overline{w_2}+\tilde R w_2 +(v_1\overline{v_2}w_1+\tilde R)\overline{r_2} )=\mathcal{O}_{L^1(M)}(1),
\end{equation}
as $s\to \infty$. 

Using \eqref{eq_5_14} and \eqref{eq_5_15}, we obtain that 
\begin{equation}
\label{eq_5_16}
\begin{aligned}
\langle A, d(u_1u_2u_3)\rangle_{g} u_4=e^{2i(\mu-L\lambda)x_1}c^{1-n}\bigg(A_1(-Ls v_1\overline{v_2}w_1\overline{w_2}
+\p_{x_1}(v_1\overline{v_2}w_1)\overline{w_2})\\
+ \langle A', d_{x'} (v_1\overline{v_2}w_1)\rangle_{g_0}\overline{w_2} \bigg)+\mathcal{O}_{L^1(M)}(1), 
\end{aligned}
\end{equation}
as $s\to \infty$.  

 It follows from \eqref{eq_2_1} in view of \eqref{eq_5_16} that 
\begin{equation}
\label{eq_5_17_0}
\begin{aligned}
\int_{M} \bigg(A_1(-Ls v_1\overline{v_2}w_1\overline{w_2}
+\p_{x_1}(v_1\overline{v_2}w_1)\overline{w_2})+ \langle A', d_{x'} (v_1\overline{v_2}w_1)\rangle_{g_0}\overline{w_2} \bigg)\\e^{2i(\mu-L\lambda)x_1}c^{1-n}dV_g=\mathcal{O}(1),
\end{aligned}
\end{equation}
as $s\to 0$. 
 
Now taking $u_2=u_3=1$ in \eqref{eq_2_1} and applying Proposition \ref{prop_boundary_A}, we obtain that $A|_{\p M}=0$ and $\p_\nu A|_{\p M}=0$. Let us extend $A$ by zero to $(\R\times M_0)\setminus M$ and denote this extension by $A$ again. Since $A\in C^{1,1}(M,T^*M)$ and $A|_{\p M}=0$, $\p_\nu A|_{\p M}=0$, we see that $A\in C^{1,1}(\R\times M_0,T^*(\R\times M_0))$. Now \eqref{eq_5_17_0} implies that 
\begin{equation}
\label{eq_5_17}
\begin{aligned}
\int_{\R\times M_0} \bigg(A_1(-Ls v_1\overline{v_2}w_1\overline{w_2}
+\p_{x_1}(v_1\overline{v_2}w_1)\overline{w_2})+ \langle A', d_{x'} (v_1\overline{v_2}w_1)\rangle_{g_0}\overline{w_2} \bigg)\\e^{2i(\mu-L\lambda)x_1}c^{1-n}dV_g=\mathcal{O}(1),
\end{aligned}
\end{equation}
as $s\to 0$. In view of \eqref{eq_5_1_3}, \eqref{eq_5_17} gives 
\begin{equation}
\label{eq_5_17_1}
\begin{aligned}
\sum_{r=1}^N \int_{\R\times U_r} \bigg(A_1(-Ls v_1\overline{v_2}w_1\overline{w_2}
+\p_{x_1}(v_1\overline{v_2}w_1)\overline{w_2})+ \langle A', d_{x'} (v_1\overline{v_2}w_1)\rangle_{g_0}\overline{w_2} \bigg)\\e^{2i(\mu-L\lambda)x_1}c^{1-n}dV_g=\mathcal{O}(1),
\end{aligned}
\end{equation}
as $s\to 0$, where $U_r$ are sufficiently small neighborhoods of the points $p_r$ of the intersections of $\gamma$ and $\eta$. Using \eqref{eq_5_2},  \eqref{eq_5_3}, \eqref{eq_5_5}, \eqref{eq_5_8}, \eqref{eq_5_10}, we obtain that in $U_r$, 
\begin{equation}
\label{eq_5_18}
\begin{aligned}
v_1\overline{v_2}w_1\overline{w_2}=s^{\frac{n-2}{2}}\sum_{1\le k,l\le P_r}\sum_{1\le m,j\le Q_r} e^{is\Psi_{klmj}^{r}}e^{\Phi_{klmj}^{r}}a_0^{(k),r}\overline{a_0^{(l),r}}c_0^{(m),r}\overline{c_0^{(j),r}}\\
+\mathcal{O}_{L^1(I\times M_0)}(1/s),
\end{aligned}
\end{equation}
where 
\begin{equation}
\label{eq_5_19}
\Psi_{klmj}^{r}=\varphi^{(k),r}-\overline{\varphi^{(l),r}}+L\psi^{(m),r}-L\overline{\psi^{(j),r}},
\end{equation}
\begin{equation}
\label{eq_5_20}
\Phi_{klmj}^{r}=-\mu \varphi^{(k),r}-\mu \overline{\varphi^{(l),r}}-L\lambda \psi^{(m),r}-L\lambda\overline{\psi^{(j),r}},
\end{equation}
and $I\subset\R$ is a bounded interval. Recall that all $a_0^{(k),r}$ and $c^{(m),r}$ are independent of $x_1$.  This fact also implies that 
\begin{equation}
\label{eq_5_21}
\p_{x_1}(v_1\overline{v_2}w_1)\overline{w_2}=\mathcal{O}_{L^1(I\times M_0)}(1/s). 
\end{equation}
 Using  \eqref{eq_5_2}, \eqref{eq_5_3}, \eqref{eq_5_5}, \eqref{eq_5_8}, \eqref{eq_5_10}, we also get in $U_r$, 
\begin{equation}
\label{eq_5_22}
\begin{aligned}
d_{x'} (v_1\overline{v_2}w_1)\overline{w_2} = s^{\frac{n-2}{2}}\sum_{1\le k,l\le P_r}\sum_{1\le m,j\le Q_r} 
is (d\varphi^{(k),r}-d\overline{\varphi^{(l),r}}+L d\psi^{(m),r}) \\
e^{is\Psi_{klmj}^{r}}e^{\Phi_{klmj}^{r}}a_0^{(k),r}\overline{a_0^{(l),r}}c_0^{(m),r}\overline{c_0^{(j),r}}+ \mathcal{O}_{L^1(I\times M_0)}(1)
\end{aligned}
\end{equation}

Substituting \eqref{eq_5_18}, \eqref{eq_5_21}, \eqref{eq_5_22} into \eqref{eq_5_17_1}, using that $dV_g=c^{\frac{n}{2}}dx_1dV_{g_0}$, and dividing \eqref{eq_5_17_1} by $s^{1/2}$, we obtain that 
\begin{equation}
\label{eq_5_23}
s^{\frac{n-1}{2}} \sum_{r=1}^N \sum_{1\le k,l\le P_r}\sum_{1\le m,j\le Q_r} \int_{U_r} B_{klmj}^{r}e^{is\Psi_{klmj}^{r}}dV_{g_0}=\mathcal{O}(s^{-1/2}).
\end{equation}
where 
\begin{equation}
\label{eq_5_24}
\begin{aligned}
B_{klmj}^r=[-L&\widehat{A_1 c^{1-\frac{n}{2}}}(2(\mu-L\lambda),\cdot)+ i \langle \widehat{A' c^{1-\frac{n}{2}}}(2(\mu-L\lambda),\cdot) , \\
&d\varphi^{(k),r}-d\overline{\varphi^{(l),r}}+L d\psi^{(m),r}\rangle_{g_0}] 
e^{\Phi_{klmj}^{r}}a_0^{(k),r}\overline{a_0^{(l),r}}c_0^{(m),r}\overline{c_0^{(j),r}}.
\end{aligned}
\end{equation}
Notice that the occurrence of the factor $s^{\frac{n-1}{2}}$ is natural here, in view of a subsequent application of the stationary phase method, in its rough version, to the integral in the left hand side of \eqref{eq_5_23}.

\subsection{Choosing $L$} The argument below follows \cite{LLLS} closely and is presented here for the completeness and convenience of the reader only.   We  claim that  $L>0$ can be chosen sufficiently large but fixed so that  $d\Psi^r_{klmj}(p_r)=0$ for all points $p_r$, $1\le r\le N$, if and only if $k=l$ and $m=j$.  Indeed, it follows from \eqref{eq_5_19} that 
\begin{equation}
\label{eq_5_25}
\begin{aligned}
\nabla \Psi^r_{klmij}(p_r)= (\nabla \varphi^{(k),r}- \nabla \overline{\varphi^{(l),r}}+L\nabla \psi^{(m),r}-L\nabla\overline{\psi^{(j),r}})(p_r)\\
=\dot{\eta}(t_k^r)- \dot{\eta}(t_l^r) +L\dot{\gamma}(\tau_m^r)-L\dot{\gamma}(\tau_j^r).
\end{aligned}
\end{equation}
If $k=l$ and $m=j$, \eqref{eq_5_25} implies that $\nabla \Psi^r_{klmij}(p_r)=0$ for all $1\le r\le N$.  Now since the geodesic $\gamma$ is non-tangential, and therefore, not closed, we have $\dot{\gamma}(\tau_m^r)-\dot{\gamma}(\tau_j^r)\ne 0$,
for all $m\ne j$ for all $r$, $1\le r\le N$.  Let 
\[
\alpha=\min\{|\dot{\gamma}(\tau_m^r)-\dot{\gamma}(\tau_j^r)|: m\ne j, 1\le m,j\le Q_r,1\le r\le N \}>0. 
\]
Then in view of the fact that $\eta$ is a unit speed geodesic, it follows from \eqref{eq_5_25} that for all $r$, $1\le r\le N$, $m\ne j$, 
\begin{equation}
\label{eq_5_26}
|\nabla \Psi^r_{klmj}(p_r)|\ge L\alpha -2\ge 1, 
\end{equation}
provided that $L\ge \frac{3}{\alpha}$.  Hence, if $d \Psi^r_{klmj}(p_r)=0$ then $m=j$, and therefore, \eqref{eq_5_25} implies that 
\begin{equation}
\label{eq_5_27}
\begin{aligned}
\nabla \Psi^r_{klmk}(p_r)= \dot{\eta}(t_k^r)-\dot{\eta}(t_l^r).
\end{aligned}
\end{equation}
This completes the proof of the claim since $\dot{\eta}(t_k^r)-\dot{\eta}(t_l^r)\ne 0$ for all $k\ne l$ and all $r$, $1\le r\le N$.  

In what follows we choose $L\ge 3/\alpha$. Furthermore, it follows from \eqref{eq_5_26} and \eqref{eq_5_27} that for such $L$, there exists $\beta>0$ such that  
\begin{equation}
\label{eq_5_28}
|\nabla \Psi^r_{klmj}(p_r)|\ge \beta>0,
\end{equation}
for  $(k, l, m,  j)\in \{(k, l, m, j): 1\le k,l\le P_r, 1\le m, j\le Q_r\}\setminus\{(k, l, m,  j): k=l, m=j\}$, $1\le r\le N$. 

Returning  to \eqref{eq_5_23}, we write the integral there as follows, 
\begin{equation}
\label{eq_5_29}
I=s^{\frac{n-1}{2}} \sum_{r=1}^N \sum_{1\le k,l\le P_r}\sum_{1\le m,j\le Q_r} \int_{U_r} B_{klmj}^{r}e^{is\Psi_{klmj}^{r}}dV_{g_0}=\sum_{r=1}^N (I_1^r+I_2^r),
\end{equation}
where 
\begin{equation}
\label{eq_5_30}
\begin{aligned}
I_1^r=s^{\frac{n-1}{2}}  \sum_{1\le k\le P_r}\sum_{1\le m\le Q_r} \int_{U_r} B_{kkmm}^{r}e^{is\Psi_{kkmm}^{r}}dV_{g_0},\\
I_2^r=s^{\frac{n-1}{2}}  \sum_{1\le k\ne l\le P_r}\sum_{1\le m\ne j\le Q_r} \int_{U_r} B_{klmj}^{r}e^{is\Psi_{klmj}^{r}}dV_{g_0}
\end{aligned}
\end{equation}

\subsection{Rough stationary phase calculation}  Here the analysis is concerned with the integrals  $I_1^r$.  It follows from \eqref{eq_5_19} that 
\[
\Psi^{r}_{kkmm}=2i(\text{Im}\, \varphi^{(k),r}+L\text{Im}\, \psi^{(m),r}),
\]
and therefore, 
$d\Psi^{r}_{kkmm}(p_r)=0$, $\Psi^{r}_{kkmm}(p_r)=0$,  and $\text{Im}\,\nabla^2 \Psi^r_{kkmm}(p_r)>0$, where $p_r\in M_0^{\text{int}}$ is the point of intersection of $\gamma$ and $\eta$.  Note that $U_r\subset M_0^{\text{int}}$, and hence, there will be no contributions from the boundary.   

Let us denote by $z=(z_1, \dots, z_{n-1})$ the geodesic normal coordinates in $(M_0,g_0)$ with origin at $p_r$. Writing $dV_{g_0}=|g_0(z)|^{1/2}dz$, applying Lemma \ref{lem_stationary_phase}, and using \eqref{eq_5_24}, \eqref{eq_5_20}, we obtain that 
\begin{equation}
\label{eq_5_31}
\begin{aligned}
\lim_{s\to \infty} s^{\frac{n-1}{2}} &  \int_{U_r} B_{kkmm}^{r}e^{is\Psi_{kkmm}^{r}}dV_{g_0}\\
= \lim_{s\to \infty}& s^{\frac{n-1}{2}}  \int_{\text{neigh}(0, \R^{n-1})} B_{kkmm}^{r}(z) |g_0(z)|^{1/2} e^{is\Psi_{kkmm}^{r}(z)}dz
=C^r_{kkmm} B_{kkmm}^{r}(p_r)\\
&=C^r_{kkmm} [-L\widehat{A_1 c^{1-\frac{n}{2}}}(2(\mu-L\lambda),p_r)+ i L \widehat{A' c^{1-\frac{n}{2}}}(2(\mu-L\lambda),p_r)(\dot{\gamma}(\tau_m^r))]\\
&e^{-2\mu t_k^r-2L\lambda\tau_m^r}|a_{00}^{(k),r}(p_r)|^2|c_{00}^{(m),r}(p_r)|^2,
\end{aligned}
\end{equation}
where 
\[
C^r_{kkmm}=\frac{(2\pi)^{\frac{n-1}{2}}}{(\det  \text{Im}\,\nabla^2 \Psi^r_{kkmm}(p_r))^{1/2} }>0.
\]
Here we also used that 
\[
\varphi^{(k),r}(p_r)=t_k^r,\quad \psi^{(m),r}(p_r)=\tau_m^r.
\]
Thus, we see from \eqref{eq_5_30} and \eqref{eq_5_31} that 
\begin{equation}
\label{eq_5_31_2}
\begin{aligned}
\lim_{s\to \infty} I_1^r= &\sum_{1\le k\le P_r}\sum_{1\le m\le Q_r}C^r_{kkmm} [-L\widehat{A_1 c^{1-\frac{n}{2}}}(2(\mu-L\lambda),p_r)\\
&+ i L \widehat{A' c^{1-\frac{n}{2}}}(2(\mu-L\lambda),p_r)(\dot{\gamma}(\tau_m^r))]
e^{-2\mu t_k^r-2L\lambda\tau_m^r}|a_{00}^{(k),r}(p_r)|^2|c_{00}^{(m),r}(p_r)|^2.
\end{aligned}
\end{equation}

\subsection{Non-stationary phase calculation} 
Here the analysis is concerned with $I_2^r$ in \eqref{eq_5_30}. It follows from \eqref{eq_5_19} that 
\begin{equation}
\label{eq_5_32}
\Psi^r_{klmj}=\tilde \Psi^r_{klmj}+ i\text{Im}\, \varphi^{(k),r}+i\text{Im}\, \varphi^{(l),r}+Li\text{Im}\, \psi^{(m),r}+ Li\text{Im}\, \psi^{(j),r},
\end{equation}
where 
\begin{equation}
\label{eq_5_33}
\tilde \Psi^r_{klmj}=\text{Re}\, \varphi^{(k),r}- \text{Re}\, \varphi^{(l),r}+L \text{Re}\, \psi^{(m),r}-L \text{Re}\, \psi^{(j),r}\in C^\infty
\end{equation}
is real such that $|\nabla \tilde \Psi^r_{klmj}(p_r)|=|\nabla  \Psi^r_{klmj}(p_r)|\ge \beta>0$  provided $L>3/\alpha$ in view of \eqref{eq_5_28}.  

Let us denote by $z=(z_1,\dots, z_{n-1})$ the geodesic normal coordinates in $(M_0,g_0)$ with origin at $p$. Motivated by \eqref{eq_5_24} and \eqref{eq_5_32}, we set 
\begin{align*}
f(z)=[-L&\widehat{A_1 c^{1-\frac{n}{2}}}(2(\mu-L\lambda),z)+ i \langle \widehat{A' c^{1-\frac{n}{2}}}(2(\mu-L\lambda),z) , \\
&d\varphi^{(k),r}-d\overline{\varphi^{(l),r}}+L d\psi^{(m),r}\rangle_{g_0}] 
e^{\Phi_{klmj}^{r}} |g_0(z)|^{1/2}\in C_0^{1,1}(M_0),
\end{align*}
and 
\begin{equation}
\label{eq_5_34_0}
\begin{aligned}
\hat a_0^{(k),r}= s^{\frac{n-2}{8}}e^{-s\text{Im}\, \varphi^{(k),r}}a_0^{(k),r}, \quad  \hat c_0^{(m),r}= s^{\frac{n-2}{8}}e^{-s\text{Im}\, \psi^{(m),r}}c_0^{(m),r}.
\end{aligned}
\end{equation}
Thus, 
\begin{equation}
\label{eq_5_34}
\begin{aligned}
I_{2, klmj}^{r}:=& s^{\frac{n-1}{2}}\int_{U_r} B_{klmj}^r e^{is \Psi^r_{klmj}}dV_g\\
&=s^{\frac{1}{2}}\int_{\text{neigh}(0,\R^{n-1})}  f(z) \hat a_0^{(k),r}\overline{\hat a_0^{(l),r}}\hat c_0^{(m),r}\overline{\hat c_0^{(j),r}}e^{is \tilde \Psi^r_{klmj}(z)}dz.
\end{aligned}
\end{equation}

Note that $f$ is independent of $s$, and 
\begin{equation}
\label{eq_5_35}
\|\hat a_0^{(k),r}\|_{L^4(M_0)}=\mathcal{O}(1), \quad \|\hat c_0^{(m),r}\|_{L^4(M_0)}=\mathcal{O}(1),
\end{equation}
as $s\to \infty$. We next  claim that 
\begin{equation}
\label{eq_5_36}
\|\nabla \hat a_0^{(k),r}\|_{L^4(M_0)}=\mathcal{O}(s^{\frac{1}{2}}), \quad \|\nabla \hat c_0^{(m),r}\|_{L^4(M_0)}=\mathcal{O}(s^{\frac{1}{2}}),
\end{equation}
as $s\to \infty$, see \cite{LLLS}. Let us recall the argument briefly. It is enough to show the first bound in \eqref{eq_5_36}. To that end, we have from \eqref{eq_5_34_0} that 
\begin{equation}
\label{eq_4_10_5}
\nabla \hat a_0^{(k),r}= s^{\frac{n-2}{8}}e^{-s\text{Im}\, \varphi^{(k),r}} ( -s (\nabla \text{Im}\, \varphi^{(k),r})a_0^{(k),r}+ \nabla a_0^{(k),r}).
\end{equation}
It suffices to control the first  term in the right hand side of \eqref{eq_4_10_5}, and to this end we note that in the Fermi coordinates $(t,y)$, associated with the geodesic $\eta$, we have 
\begin{equation}
\label{eq_4_11_5}
|\nabla \text{Im}\, \varphi^{(k),r}(t,y)|=\mathcal{O}(|y|),
\end{equation}
 and 
\begin{equation}
\label{eq_4_12_5}
\text{Im}\,\varphi^{(k),r}(t,y)\ge c|y|^2,
\end{equation}
for some $c>0$, see \eqref{eq_5_4}. 
Thus, using \eqref{eq_4_11_5}  and \eqref{eq_4_12_5}, we get
\begin{align*}
\|s^{\frac{n-2}{8}}e^{-s\text{Im}\, \varphi^{(k),r}} & s (\nabla \text{Im}\, \varphi^{(k),r})a_0^{(k),r} \|_{L^4(M_0)}\\
&=\mathcal{O}(s^{\frac{n-2}{8}}s)\bigg(\int_{|y|\le 1/2}
e^{-4s\text{Im}\, \varphi^{(k),r}} |y|^4 dy \bigg)^{1/4}=\mathcal{O}(s^{\frac{1}{2}}).
\end{align*}
This bound together with \eqref{eq_4_10_5} shows the first bound in \eqref{eq_5_36}.
Similarly to \eqref{eq_5_36}, we also have 
\begin{equation}
\label{eq_5_37}
\|\p^\alpha \hat a_0^{(k),r}\|_{L^4(M_0)}=\mathcal{O}(s^{\frac{|\alpha|}{2}}),\quad \|\p^\alpha \hat c_0^{(m),r}\|_{L^4(M_0)}=\mathcal{O}(s^{\frac{|\alpha|}{2}}),\quad \forall \alpha, 
\end{equation}
as $s\to \infty$.  Furthermore, as $\delta'>0$ can be chosen as small as we wish, we see that $\hat a_0^{(k),r}$, $\hat c_0^{(k),r}$ have compact support in $U_r$. 

Letting 
\[
L=\frac{\nabla \tilde  \Psi^r_{klmij}\cdot \nabla}{i|\nabla \tilde  \Psi^r_{klmij}|^2},
\]
we have $L(e^{is \tilde \Psi^r_{klmij}})= se^{is \tilde  \Psi^r_{klmij}}$. Integrating by parts in \eqref{eq_5_34}, we get 
\[
I_{2,klmj}^r=s^{-\frac{1}{2}} \int_{\text{neigh}(0,\R^{n-1})} e^{is \tilde \Psi^r_{klmj}(z)}
L^t(  f(z) \hat a_0^{(k),r}\overline{\hat a_0^{(l),r}}\hat c_0^{(m),r}\overline{\hat c_0^{(j),r}} )dz,
\]
where $L^t=-L-\div L$. Now in view of \eqref{eq_5_33} and \eqref{eq_5_35}, we see that 
\[
s^{-\frac{1}{2}}\bigg| \int_{\text{neigh}(0,\R^{n-1})} e^{is \tilde \Psi^r_{klmj}(z)}
(\div L) (  f(z) \hat a_0^{(k),r}\overline{\hat a_0^{(l),r}}\hat c_0^{(m),r}\overline{\hat c_0^{(j),r}} )dz\bigg|=\mathcal{O}(s^{-\frac{1}{2}}),
\]
and in view of \eqref{eq_5_36}, 
\[
s^{-\frac{1}{2}}\bigg| \int_{\text{neigh}(0,\R^{n-1})} e^{is \tilde \Psi^r_{klmj}(z)}
   f(z) \nabla (\hat a_0^{(k),r}\overline{\hat a_0^{(l),r}}\hat c_0^{(m),r}\overline{\hat c_0^{(j),r}} )dz\bigg|=\mathcal{O}(1),
\]
as $s\to \infty$. As $f$ is independent of $s$, we see, after one integration by parts in \eqref{eq_5_34}, that $I_{2,klmj}^r=\mathcal{O}(1)$. Since $\nabla f$ is Lipschitz,  we can integrate by parts the second time and using \eqref{eq_5_37}, we get 
\begin{equation}
\label{eq_5_38}
I_{2,klmj}^r=\mathcal{O}(s^{-\frac{1}{2}}),
\end{equation}
as $s\to \infty$.  Notice that it is precisely here that we need that our 1-form $A\in C_0^{1,1}(\R\times M_0,T^*(\R\times M_0))$. 

We get, in view of \eqref{eq_5_30}, \eqref{eq_5_38},
\begin{equation}
\label{eq_5_39}
I_{2}^r=\mathcal{O}(s^{-\frac{1}{2}}),
\end{equation}
as $s\to \infty$.

\subsection{Completion of the proof} Passing to the limit $s\to \infty$ in \eqref{eq_5_23} and using \eqref{eq_5_29}, \eqref{eq_5_30}, \eqref{eq_5_31_2}, and \eqref{eq_5_39}, we obtain that 
\begin{equation}
\label{eq_5_40}
\begin{aligned}
\sum_{r=1}^N \sum_{k=1}^{P_r}\sum_{m=1}^{Q_r}C^r_{kkmm}  [-L\widehat{A_1 c^{1-\frac{n}{2}}}(2(\mu-L\lambda),p_r)+ i L \widehat{A' c^{1-\frac{n}{2}}}(2(\mu-L\lambda),p_r)(\dot{\gamma}(\tau_m^r))]\\
e^{-2\mu t_k^r-2L\lambda\tau_m^r}|a_{00}^{(k),r}(p_r)|^2|c_{00}^{(m),r}(p_r)|^2=0.
\end{aligned}
\end{equation}
Next we would like to determine each term in the sum in \eqref{eq_5_40} separately. To do this, we shall follow \cite{LLLS}. First choosing $\mu=(1-L)\lambda$, we get  
\begin{equation}
\label{eq_5_41}
\begin{aligned}
\sum_{r=1}^N \sum_{k=1}^{P_r}\sum_{m=1}^{Q_r} [-L\widehat{A_1 c^{1-\frac{n}{2}}}(2\lambda(1-2L),p_r)
+ i L \widehat{A' c^{1-\frac{n}{2}}}(2\lambda(1-2L),p_r)(\dot{\gamma}(\tau_m^r))]\\
C^r_{kkmm} e^{-2\lambda [L ( \tau_m^r -  t_k^r) +t_k^r]}|a_{00}^{(k),r}(p_r)|^2|c_{00}^{(m),r}(p_r)|^2=0.
\end{aligned}
\end{equation}
It is shown in  \cite{LLLS}, that for all $L\ge 1$ sufficiently large, 
 \begin{equation}
\label{eq_5_41_2}
 L(\tau_{m_1}^{r_1}- t_{k_1}^{r_1})+t_{k_1}^{r_1}\ne L(\tau_{m_2}^{r_2}- t_{k_2}^{r_2})+t_{k_2}^{r_2},
\end{equation}
when $(r_1,k_1, m_1)\ne (r_2,k_2, m_2)$, and fixing $L\ge \frac{3}{\alpha}$ large enough, we may assume in what follows that \eqref{eq_5_41_2} holds. We shall next need Lemma 5.2 from \cite{LLLS} which can be stated as follows: let $f_1,\dots, f_N \in \mathcal{E}'(\R)$ be such that for some distinct real numbers $a_1,\dots, a_N$, one has 
\[
\sum_{j=1}^N\hat f_j(\lambda)e^{a_j\lambda}=0, \quad \lambda \in \R, 
\]
then $f_1=\dots=f_N=0$.  Applying this result, we get for all $r$, $k$, $m$, $\lambda$, 
\begin{align*}
(-\widehat{A_1 c^{1-\frac{n}{2}}}(2\lambda(1-2L),p_r)
+ i \widehat{A' c^{1-\frac{n}{2}}}(2\lambda(1-2L),p_r)(\dot{\gamma}(\tau_m^r)))\\
C^r_{kkmm} |a_{00}^{(k),r}(p_r)|^2|c_{00}^{(m),r}(p_r)|^2=0,
\end{align*}
and as $C^r_{kkmm}\ne 0$,  $a_{00}^{(k),r}(p_r)\ne 0$, and $c_{00}^{(m),r}(p_r)\ne 0$, we get taking the inverse Fourier  transform in $x_1$, 
\[
- A_1 (x_1,p_r)+ i A' (x_1,p_r)(\dot{\gamma}(\tau_m^r))=0,
\]
for all $x_1\in \R$, $p_r$, and $\tau_m^r$. Since $y_0$ was one of the points $p_r$, and $\gamma(\tau_m^r)=y_0$, we know 
\begin{equation}
\label{eq_5_42}
(i A_1, A')(x_1,y_0)(1, \dot{\gamma}(\tau^r_m))=0. 
\end{equation}
Here we may replace $\dot{\gamma}(\tau^r_m)$ by $-\dot{\gamma}(\tau^r_m)$, and thus, \eqref{eq_5_42} implies that $A_1(x_1,y_0)=0$, for all $x_1\in \R$ and almost all $y_0\in M_0$, and therefore, by continuity,  $A_1\equiv 0$.  Hence, we are left with the recovery of the $1$-form $A'(x_1,\cdot)$ on $M_0$ from the fact that 
\begin{equation}
\label{eq_5_43}
A'(x_1,y_0)(\dot{\gamma}(\tau^r_m))=0.
\end{equation}
To proceed we assume without loss of generality that $v_1:=\dot{\gamma}(\tau^r_m)=(1,0, \dots, 0)\in \R^{n-1}$, and   consider its small perturbations $v_2, \dots, v_{n-1}$ given by \eqref{eq_4_22-new}.  The unit vectors $v_1, \dots, v_{n-1}$ are linearly independent, and therefore, they span the tangent space $T_{y_0} M_0$.  By Proposition \ref{prop_geodesic}, for $\varepsilon>0$ sufficiently small, the unit speed geodesic $\gamma_{y_0,v_j}$, $j=2,\dots, n-1$, through $(y_0,v_j)$ is non-tangential between boundary points, and $\gamma$ and $\gamma_{y_0,v_j}$ do not intersect each other at the boundary of $M_0$.  Applying the discussion above with $\eta=\gamma$ and $\gamma=\gamma_{y_0,v_j}$, we get 
\begin{equation}
\label{eq_5_45}
A'(x_1,y_0)(v_j)=0, \quad j=2,\dots, n-1.
\end{equation}
It follows from \eqref{eq_5_43} and \eqref{eq_5_45} that the $1$-form $A'(x_1,y_0)=0$, and therefore,  $A'\equiv 0$. 
This completes the proof of Proposition \ref{prop_density_form} in the general setting.

In the course of the proof of Proposition \ref{prop_density_form}, we also proved the following result. 

\begin{cor}
\label{cor_density_form}
Let $(M,g)$ be a conformally transversally anisotropic manifold of dimension $n\ge 3$. Let $A\in C^{1,1}(M,T^*M)$ be a $1$-form such that $A|_{\p M}=0$ and $\p_\nu A|_{\p M}=0$. If 
\[
\int_M  \langle A, d(u_1u_2u_3)\rangle_g u_4dV_g=\mathcal{O}(1),
\]
as  $s\to \infty$, for all $u_l\in C^3(M)$, $l=1,\dots, 4$, harmonic of the form \eqref{eq_5_1}, then $A\equiv 0$. 
\end{cor}

\section{Proof of Theorem \ref{thm_main}}

\label{sec_proof_thm_main}

Let $\varepsilon=(\varepsilon_1,\dots, \varepsilon_m)\in \C^m$, $m\ge 3$, and consider the Dirichlet problem \eqref{eq_int_3} with
\[
f=\sum_{k=1}^m \varepsilon_k f_k,\quad f_k\in C^{2,\alpha}(\p M), \quad k=1,\dots, m. 
\] 
Then for all $|\varepsilon|$ sufficiently small, the problem \eqref{eq_int_3}  has a unique small solution $u(\cdot,\varepsilon)\in C^{2,\alpha}(M)$, which depends holomorphically on $\varepsilon\in \text{neigh}(0, \C^m)$, see Theorem \ref{thm_well-posedness}. 

We shall use an induction argument on $m\ge 3$ to show that all the coefficients $A_m$ and $V_m$ in \eqref{eq_int_2_A}, \eqref{eq_int_2_V}, see also \eqref{eq_int_3}, can be determined from the Dirichlet--to--Neumann map $\Lambda_{A,V}$ given in \eqref{eq_int_3_DN}. 

First, let $m=3$, and let us proceed to carry out a third order linearization of the Dirichlet--to--Neumann map. Let $u_j=u_j(x,\varepsilon)$ be the unique small solution of the Dirichlet problem,
\begin{equation}
\label{eq_6_1}
\begin{cases}
-\Delta_g u_j+id^*(\sum_{k=2}^\infty A^{(j)}_k(x)\frac{u_j^k}{k!} u_j)-i \langle \sum_{k=2}^\infty A^{(j)}_k(x)\frac{u_j^k}{k!}, du_j\rangle_g\\
 \, \quad   + \langle \sum_{k=2}^\infty A^{(j)}_k(x)\frac{u_j^k}{k!}, \sum_{k=2}^\infty A^{(j)}_k(x)\frac{u_j^k}{k!}\rangle_g u_j 
 +\sum_{k=3}^\infty V^{(j)}_k(x)\frac{u_j^k}{k!}=0 \text{ in }M,\\
u_j=\varepsilon_1f_1+\varepsilon_2f_2+ \varepsilon_3f_3  \text{ on }\p M,
\end{cases}
\end{equation}
for $j=1,2$. Differentiating \eqref{eq_6_1} with respect to $\varepsilon_l$, $l=1,2,3$, and using that $u_j(x,0)=0$, we get 
\begin{equation}
\label{eq_6_2}
\begin{cases}
-\Delta_g v_j^{(l)}=0 & \text{in } M,\\
v_j^{(l)}=f_l & \text{on }\p M,
\end{cases}
\end{equation}
where $v_j^{(l)}=\p_{\varepsilon_l}u_j|_{\varepsilon=0}$. By the uniqueness and the elliptic regularity for the Dirichlet problem \eqref{eq_6_2}, we have that $v^{(l)}:=v_1^{(l)}=v_2^{(l)}\in C^{2,\alpha} (M)$, $l=1,2,3$, see \cite[Theorem 6.15]{Gil_Tru_book}. 

Applying $\p_{\varepsilon_k}\p_{\varepsilon_l}|_{\varepsilon=0}$, $k,l=1,2,3$, to \eqref{eq_6_1}, we next get 
\begin{equation}
\label{eq_6_3}
\begin{cases}
-\Delta_g w_j^{(k,l)}=0 & \text{in } M,\\
w_j^{(k,l)}=0 & \text{on }\p M,
\end{cases}
\end{equation}
where $w_j^{(k,l)}=\p_{\varepsilon_k}\p_{\varepsilon_l}u_j|_{\varepsilon=0}$, and therefore, $w_j^{(k,l)}=0$ for all $j=1,2$, $k,l=1,2,3$. 
Finally, applying $\p_{\varepsilon_1}\p_{\varepsilon_2}\p_{\varepsilon_3}|_{\varepsilon=0}$ to \eqref{eq_6_1}, we obtain the third order linearization,
\begin{equation}
\label{eq_6_4}
\begin{cases}
-\Delta_g w_j +3i d^*(A_2^{(j)}v^{(1)}v^{2}v^{(3)})- i \langle A_2^{(j)}, d(v^{(1)}v^{(2)}v^{(3)})\rangle_g \\
\quad \quad \quad +V^{(j)}_3 v^{(1)}v^{(2)}v^{(3)} =0  \text{ in } M,\\
w_j=0  \text{ on }\p M,
\end{cases}
\end{equation}
where $w_j=\p_{\varepsilon_1}\p_{\varepsilon_2}\p_{\varepsilon_3}u_j|_{\varepsilon=0}$. Using that 
\begin{equation}
\label{eq_6_4_2}
d^*(Av)=(d^*A)v -\langle A, dv\rangle_g,
\end{equation}
for any $1$-form $A$ and a function $v$, we can rewrite  \eqref{eq_6_4} as follows, 
\begin{equation}
\label{eq_6_5}
\begin{cases}
-\Delta_g w_j - 4 i \langle A_2^{(j)}, d(v^{(1)}v^{(2)}v^{(3)})\rangle_g+ \big(3i d^*(A_2^{(j)}) +V_3^{(j)}\big) v^{(1)}v^{(2)}v^{(3)}=0 \text{ in } M,\\
w_j=0  \text{ on }\p M.
\end{cases}
\end{equation}
The fact that $\Lambda_{A^{(1)}, V^{(1)}}(\varepsilon_1f_1+\varepsilon_2f_2+ \varepsilon_3f_3 )=\Lambda_{A^{(2)}, V^{(2)}}(\varepsilon_1f_1+\varepsilon_2f_2+ \varepsilon_3f_3 )$ for all small $\varepsilon$ and all $f_j\in C^{2,\alpha}(\p M)$
implies that $\p_\nu u_1|_{\p M}= \p_\nu u_2|_{\p M}$. Therefore, an application of $\p_{\varepsilon_1}\p_{\varepsilon_2}\p_{\varepsilon_3}|_{\varepsilon=0}$  yields $\p_\nu w_1|_{\p M}= \p_\nu w_2|_{\p M}$. 
Multiplying \eqref{eq_6_5} by $v^{(4)}\in C^{2,\alpha}(M)$ harmonic in $(M,g)$ and applying Green's formula, we get 
\begin{equation}
\label{eq_6_6}
\int_M \big(4i \langle A, d(v^{(1)}v^{(2)}v^{(3)})\rangle_g v^{(4)}-\big(3i d^*(A) +V\big) v^{(1)}v^{(2)}v^{(3)} v^{(4)} \big)dV_g=0,
\end{equation}
for all $v^{(l)}\in C^{2,\alpha}(M)$ harmonic in $(M,g)$, $l=1,\dots, 4$. Here $A=A_2^{(1)}-A_2^{(2)}$ and $V=V_3^{(1)}-V_3^{(2)}$.  An application of Proposition \ref{prop_boundary_A_V} implies that $A|_{\p M}=0$ and $\p_\nu A|_{\p M}=0$. 

Choosing $v^{(l)}=u_l\in C^3(M)$, $l=1,\dots, 4$, to be harmonic functions of the form \eqref{eq_5_1}, and using \eqref{eq_5_15_2}, we first observe that \eqref{eq_6_6} implies that 
\[
\int_M  \langle A, d(u_1u_2u_3)\rangle_g u_4dV_g=\mathcal{O}(1),
\]
as $s\to \infty$.  By Corollary \ref{cor_density_form},  we get $A\equiv 0$, and therefore, $A_2^{(1)}=A_2^{(2)}$. Substituting $A= 0$ into  \eqref{eq_6_6}, we get 
\[
\int_M V v^{(1)}v^{(2)}v^{(3)} v^{(4)} dV_g=0,
\]
for all harmonic functions $v^{(l)}\in C^{2,\alpha}(M)$, $l=1,\dots, 4$.  Using Proposition \ref{prop_density_potential}, we obtain that $V=0$, and thus, $V_3^{(1)}=V_3^{(2)}$.

Let $m\ge 4$ and assume that 
\begin{equation}
\label{eq_6_7_-1}
A_k=A_k^{(1)}=A_k^{(2)}, \quad k=2, \dots, m-2,\quad  V_k=V_k^{(1)}=V_k^{(2)}, \quad  k=3,\dots, m-1.
\end{equation} 
To show that $A_{m-1}^{(1)}=A_{m-1}^{(2)}$ and $V_m^{(1)}=V_m^{(2)}$, we shall perform the $m$th order linearization of the Dirichlet--to--Neumann map.  To that end, let $u_j=u_j(x,\varepsilon)$ be the unique small solution of the Dirichlet problem,
\begin{equation}
\label{eq_6_7}
\begin{cases}
-\Delta_g u_j+id^*(\sum_{k=2}^\infty A^{(j)}_k(x)\frac{u_j^k}{k!} u_j)-i \langle \sum_{k=2}^\infty A^{(j)}_k(x)\frac{u_j^k}{k!}, du_j\rangle_g\\
 \, \quad   + \langle \sum_{k=2}^\infty A^{(j)}_k(x)\frac{u_j^k}{k!}, \sum_{k=2}^\infty A^{(j)}_k(x)\frac{u_j^k}{k!}\rangle_g u_j 
 +\sum_{k=3}^\infty V^{(j)}_k(x)\frac{u_j^k}{k!}=0 \text{ in }M,\\
u_j=\varepsilon_1f_1+\dots+ \varepsilon_mf_m  \text{ on }\p M,
\end{cases}
\end{equation}
for $j=1,2$.  We would like to apply $\p_{\varepsilon_1}\dots\p_{\varepsilon_m}|_{\varepsilon=0}$ to \eqref{eq_6_7}.  First we observe that 
\[
\p_{\varepsilon_1}\dots\p_{\varepsilon_m} \bigg(id^*(\sum_{k=m}^\infty A^{(j)}_k(x)\frac{u_j^k}{k!} u_j)-i \langle \sum_{k=m}^\infty A^{(j)}_k(x)\frac{u_j^k}{k!}, du_j\rangle_g + \sum_{k=m+1}^\infty V^{(j)}_k(x)\frac{u_j^k}{k!}\bigg)
\]
is a sum of terms, each of them containing positive powers of $u_j$ which vanish when $\varepsilon=0$. The only term in 
the expression for $\p_{\varepsilon_1}\dots\p_{\varepsilon_m} (V^{(j)}_m(x)\frac{u_j^m}{m!})$ which does not contain a positive power of $u_j$ is $V^{(j)}_m(x) \p_{\varepsilon_1}u_j \cdots \p_{\varepsilon_m}u_j$.  Furthermore, the only term in 
the expression for $\p_{\varepsilon_1}\dots\p_{\varepsilon_m} (id^*(A^{(j)}_{m-1}\frac{u_j^m}{(m-1)!} ) )$ which does not contain a positive power of $u_j$ is $mi d^*(A_{m-1}^{(j)} \p_{\varepsilon_1}u_j \cdots \p_{\varepsilon_m}u_j)$. The only terms in $\p_{\varepsilon_1}\dots\p_{\varepsilon_m}  \langle A^{(j)}_{m-1}\frac{u_j^{m-1}}{(m-1)!}, du_j\rangle_g$ which do not contain a positive power of $u_j$ can be written as $ \langle A^{(j)}_{m-1}, d(\p_{\varepsilon_1}u_j \cdots \p_{\varepsilon_m}u_j)\rangle_g$. The expression 
\[
\p_{\varepsilon_1}\dots\p_{\varepsilon_m} \bigg(id^*(\sum_{k=2}^{m-2}A^{(j)}_k(x)\frac{u_j^k}{k!} u_j)-i \langle \sum_{k=2}^{m-2} A^{(j)}_k(x)\frac{u_j^k}{k!}, du_j\rangle_g + \sum_{k=3}^{m-1} V^{(j)}_k(x)\frac{u_j^k}{k!}\bigg)
\]
is independent of $j=1,2$, in view of \eqref{eq_6_7_-1} and the fact that it contains only derivatives of $u_j$ of the form  $\p^k_{\varepsilon_{l_1}, \dots, \varepsilon_{l_k}}u_j|_{\varepsilon=0}$ with $k=1, \dots, m-2$, $\varepsilon_{l_1}, \dots, \varepsilon_{l_k}\in \{ \varepsilon_1, \dots, \varepsilon_m\}$.  Here we use the fact that $\p^k_{\varepsilon_{l_1}, \dots, \varepsilon_{l_k}}u_1|_{\varepsilon=0}= \p^k_{\varepsilon_{l_1}, \dots, \varepsilon_{l_k}}u_2|_{\varepsilon=0}$ for  $k=1, \dots, m-1$, $\varepsilon_{l_1}, \dots, \varepsilon_{l_k}\in \{ \varepsilon_1, \dots, \varepsilon_m\}$. This follows by applying the operators $\p^k_{\varepsilon_{l_1}, \dots, \varepsilon_{l_k}}|_{\varepsilon=0}$ to \eqref{eq_6_7}, using \eqref{eq_6_7_-1}, and the unique solvability of the Dirichlet problem for the Laplacian. 

The terms in the expression for 
\[
\p_{\varepsilon_1}\dots\p_{\varepsilon_m} \bigg( \langle \sum_{k=2}^\infty A^{(j)}_k(x)\frac{u_j^k}{k!}, \sum_{k=2}^\infty A^{(j)}_k(x)\frac{u_j^k}{k!}\rangle_g u_j \bigg)
\]
which do  not contain a positive power of $u_j$, only contain $A_2^{(j)}, \dots, A_{m-3}^{(j)}$, and only derivatives of $u_j$ of the form  $\p^k_{\varepsilon_{l_1}, \dots, \varepsilon_{l_k}}u_j|_{\varepsilon=0}$ with $k=1, \dots, m-4$, $\varepsilon_{l_1}, \dots, \varepsilon_{l_k}\in \{ \varepsilon_1, \dots, \varepsilon_m\}$, which are  independent of $j=1,2$. 

Hence, the $m$th order linearization has the following form,
\begin{equation}
\label{eq_6_8}
\begin{cases}
-\Delta_g w_j +m i d^*(A_{m-1}^{(j)}v^{(1)}\cdots v^{(m)})- i \langle A_{m-1}^{(j)}, d(v^{(1)}\cdots v^{(m)})\rangle_g \\
\quad \quad \quad +V^{(j)}_m v^{(1)}\cdots v^{(m)} =H_m  \text{ in } M,\\
w_j=0  \text{ on }\p M,
\end{cases}
\end{equation}
where $w_j=\p_{\varepsilon_1}\dots\p_{\varepsilon_m}u_j|_{\varepsilon=0}$, and $H_m$ is known and independent of $j=1,2$.  Using \eqref{eq_6_4_2}, \eqref{eq_6_8} can be written as follows, 
\begin{equation}
\label{eq_6_9}
\begin{cases}
-\Delta_g w_j - (m+1) i \langle A_{m-1}^{(j)}, d(v^{(1)}\cdots v^{(m)})\rangle_g\\
\quad \quad \quad + \big(m i d^*(A_{m-1}^{(j)}) +V_m^{(j)}\big) v^{(1)}\cdots v^{(m)}=H_m \text{ in } M,\\
w_j=0  \text{ on }\p M.
\end{cases}
\end{equation}
Proceeding as in the case $m=3$, we see that 
\begin{equation}
\label{eq_6_10}
\int_M \big((m+1)i \langle A, d(v^{(1)}\cdots v^{(m)})\rangle_g v^{(m+1)}-\big(mi d^*(A) +V\big) v^{(1)}\cdots v^{(m+1)} \big)dV_g=0,
\end{equation}
for any $v^{(l)}\in C^{2,\alpha}(M)$ harmonic, $l=1,\dots, m+1$. Here $A=A_{m-1}^{(1)}-A_{m-1}^{(2)}$ and $V=V_{m}^{(1)}-V_m^{(2)}$. 
Setting $v^{(1)}=\dots=v^{(m-3)}=1$ and arguing as in the case $m=3$, we complete the proof of Theorem \ref{thm_main}.

\begin{appendix}

\section{A rough stationary phase argument}

\label{sec_rough_stationary_phase}

We need the following rough version of the stationary phase, see \cite{LLLS}.

\begin{lem} 
\label{lem_stationary_phase}
Let $\Psi\in C^\infty(\R^n;\R)$ be such that 
\begin{equation}
\label{eq_ap_1}
\Psi(0)=0, \quad  \Psi'(0)=0, \quad \Psi''(0)>0.
\end{equation}
Let $V\subset \R^n$ be a sufficiently small neighborhood of zero, and $a\in C(\overline{V})$. Then 
\begin{equation}
\label{eq_ap_2}
\lim_{s\to \infty}s^{\frac{n}{2}}\int_V e^{-s\Psi(z)}a(z)dz=\frac{(2\pi)^{\frac{n}{2}}}{(\emph{\det}  \Psi''(0))^{1/2} }a(0). 
\end{equation}
\end{lem}

\begin{proof}
Taylor expanding the phase function $\Psi$ and using \eqref{eq_ap_1}, we get 
\[
\Psi(z)=\frac{1}{2}\Psi''(0)z\cdot z+\mathcal{O}(|z|^3),
\]
and therefore, 
\begin{equation}
\label{eq_ap_3}
\Psi(z)\ge c|z|^2,
\end{equation}
with some $c>0$, for all $z\in V$, a sufficiently small neighborhood of zero. Making the change of variables $z\mapsto s^{1/2}z$ in the integral in \eqref{eq_ap_2} and using the dominated convergence theorem, we obtain that 
\begin{align*}
\lim_{s\to \infty}s^{\frac{n}{2}}\int_V e^{-s\Psi(z)}a(z)dz=\lim_{s\to \infty} \int_{s^{1/2}V}e^{-s\Psi(z/s^{1/2})}a(z/s^{1/2})dz\\
=\bigg(\int_{\R^n} e^{-\frac{1}{2}\Psi''(0)z\cdot z}dz\bigg)a(0)=\frac{(2\pi)^{\frac{n}{2}}}{(\det \Psi''(0))^{1/2} }a(0). 
\end{align*}
Here we use the following consequence of \eqref{eq_ap_3},
\[
|\chi_{s^{1/2}V}e^{-s\Psi(z/s^{1/2})}a(z/s^{1/2})|\le \mathcal{O}(1)e^{-c|z|^2}\in L^1(\R^n),
\]
where $\chi_{s^{1/2}V}$ is the characteristic function of the set $s^{1/2}V$. Thus, \eqref{eq_ap_2} follows. 
\end{proof}

\section{Well-posedness of the Dirichlet problem for a nonlinear magnetic Schr\"odinger equation}

\label{sec_well-posedness}

The purpose of this appendix is to show the well-posedness of  the Dirichlet problem for a nonlinear magnetic Schr\"odinger equation with small boundary data. The argument is standard, see \cite{LLLS}, \cite{Krup_Uhlmann_2}, and is given here for completeness and convenience of the reader.

Let $(M,g)$ be a smooth compact Riemannian manifold of dimension $n\ge 2$ with smooth boundary. 
Let $C^{k,\alpha}(M)$ stand for the H\"older space on $M$, where $k\in \N\cup \{0\}$ and $0<\alpha<1$, see \cite[Appendix A]{Hormander_1976}. 
Let us note that $C^{k,\alpha}(M)$ is an  algebra under pointwise multiplication, and
\begin{equation}
\label{eq_ap2_1}
\|uv\|_{C^{k, \alpha}(M)} \le C\big(\|u\|_{C^{k,\alpha}(M)} \|v\|_{L^\infty(M)}+\|u\|_{L^\infty(M)} \|v\|_{C^{k,\alpha}M)} \big), \quad u,v\in C^{k,\alpha}(M),
\end{equation}
see \cite[Theorem A.7]{Hormander_1976}. 

Consider the Dirichlet problem for the nonlinear magnetic Schr\"odinger operator, 
\begin{equation}
\label{eq_ap2_2}
\begin{cases} L_{A,V}u=0 & \text{in }M, \\
u=f & \text{on } \p M,
\end{cases}
\end{equation}
where $L_{A,V}$ is given in \eqref{eq_int_3_L_AV}. 
Here the $1$-form $A: M\times \C\to T^*M$ and the function $V: M\times \C\mapsto \C$ satisfy the conditions: 
\begin{itemize}
\item[($A$)] the map $\C\ni z\mapsto A(\cdot, z)$ is holomorphic with values in $C^{1,\alpha}(M,T^*M)$, 
the space of $1$-forms with complex valued  $C^{1,\alpha}(M)$ coefficients,
\item[($V_{i}$)] the map  $\C\ni z\mapsto V(\cdot, z)$ is holomorphic with values in $C^{\alpha}(M)$, 
\item[($V_{ii}$)] $V(x,0)=0, \text{ for all } x\in M$. 
\end{itemize}

The condition ($V_{ii}$) guaranties that $u=0$ is a solution to \eqref{eq_ap2_2} when $f=0$. 
It follows from ($A$), ($V_i$), ($V_{ii}$)  that $A$ and $V$ can be expanded into power series 
\begin{equation}
\label{eq_ap2_3}
A(x,z)=\sum_{k=0}^\infty A_k(x)\frac{z^k}{k!}, \quad A_k(x):=\p_z^k A(x,0)\in C^{1,\alpha}(M,T^*M),
\end{equation}
converging in the $C^{1,\alpha}(M,T^*M)$ topology, and 
\begin{equation}
\label{eq_ap2_4}
V(x,z)=\sum_{k=1}^\infty V_k(x)\frac{z^k}{k!}, \quad V_k(x):=\p_z^k V(x,0)\in C^{\alpha}(M),
\end{equation}
converging in the $C^{\alpha}(M)$ topology. We also assume that $A_0\in C^\infty(M,T^*M)$ and $V_1\in C^\infty(M)$. Let us assume furthermore that 

\begin{itemize}
\item[(i)] $0$ is not a Dirichlet eigenvalue of the operator $d^*_{\overline{A_0}}d_{A_0}+V_1$.
\end{itemize}

Under all of the assumptions above, we have the following result. 
\begin{thm}
\label{thm_well-posedness}
There exist $\delta>0$, $C>0$ such that for any $f\in  B_{\delta}(\p M):=\{f\in C^{2,\alpha}(\p M): \|f\|_{C^{2,\alpha}(\p M)}< \delta\}$, the problem \eqref{eq_ap2_2} has a solution $u=u_f\in C^{2,\alpha}(M)$ which satisfies
\[
\|u\|_{C^{2,\alpha}(M)}\le C\|f\|_{C^{2,\alpha}(\p M)}.
\]
The solution $u$ is unique within the class $\{u\in C^{2,\alpha}(M): \|u\|_{C^{2,\alpha}(M)}< C\delta \}$ and
 it depends holomorphically on $f\in B_\delta(\p M)$. Furthermore, the map
\[
B_\delta(\p M)\to C^{1,\alpha}(M), \quad f\mapsto \p_\nu u_f|_{\p M}
\]
is holomorphic.
\end{thm}

\begin{proof}
We shall follow \cite{LLLS}, see also \cite{Krup_Uhlmann_2},  and in order to prove this result we shall rely on the implicit function theorem for holomorphic maps between complex Banach spaces, see \cite[p. 144]{Poschel_Trub_book}. To that end,  we let 
\[
B_1=C^{2,\alpha}(\p M), \quad  B_2=C^{2,\alpha}(M), \quad B_3=C^{\alpha}(M)\times C^{2,\alpha}(\p M),
\]
and introduce the map, 
\begin{equation}
\label{eq_ap2_5}
F:B_1\times B_2\to B_3, \quad F(f, u)=(L_{A,V}u, u|_{\p M}-f). 
\end{equation}
Let us verify that the map $F$ indeed enjoys the mapping properties given in \eqref{eq_ap2_5}. To that end, let $u\in C^{2,\alpha}(M)$ and note first that $-\Delta_g u\in C^{\alpha}(M)$. Let us check that $A(\cdot, u(\cdot))\in C^{1,\alpha}(M, T^*M)$.  By Cauchy's estimates, the coefficients $A_k$ in \eqref{eq_ap2_3} satisfy 
\begin{equation}
\label{eq_ap2_6}
\|A_k\|_{C^{1,\alpha}(M, T^*M)}\le \frac{k!}{R^k}\sup_{|z|=R}\|A(\cdot, z)\|_{C^{1,\alpha}(M, T^*M)}, \quad R>0,
\end{equation}
for all $k=0,1,\dots$.
Using \eqref{eq_ap2_1} and \eqref{eq_ap2_6}, we obtain that 
\begin{equation}
\label{eq_ap2_7}
\bigg\| \frac{A_k}{k!}u^k \bigg\|_{C^{1,\alpha}(M, T^*M)}\le \frac{C^k}{R^k}\|u\|^k_{C^{1,\alpha}(M)}\sup_{|z|=R}\|A(\cdot, z)\|_{C^{1,\alpha}(M, T^*M)},
\end{equation}
for all $k=0,1,\dots$. Choosing $R=2C\|u\|_{C^{1,\alpha}(M)}$, it follows from \eqref{eq_ap2_7} that the series 
$\sum_{k=0}^\infty A_k(x)\frac{u^k}{k!}$ converges in $C^{1,\alpha}(M, T^*M)$, and therefore, $A(\cdot, u(\cdot))\in C^{1,\alpha}(M, T^*M)$.  Similarly, $V(\cdot ,u(\cdot))\in C^\alpha(M)$, see also \cite{Krup_Uhlmann_2}. Hence, using \eqref{eq_int_3_L_AV}, we see that $L_{A,V}u\in C^{\alpha}(M)$. 

We next claim that the map $F$ in \eqref{eq_ap2_5} is holomorphic. To this end, we first note that $F$ is locally bounded as $F$ is continuous in $(f,u)$.  Thus, it suffices to show that $F$
 is weakly holomorphic, see \cite[p. 133]{Poschel_Trub_book}. In doing so, let $(f_0,u_0), (f_1,u_1)\in B_1\times B_2$, and let us prove that the map
\[
\lambda\mapsto F((f_0,u_0)+\lambda(f_1,u_1))
\]
is holomorphic in $\C$ with valued in $B_3$. It suffices to check that the map $\lambda\mapsto A(x, u_0(x)+\lambda u_1(x))$ is holomorphic in $\C$ with values in $C^{1,\alpha}(M, T^*M)$, as the fact that the map $\lambda\mapsto V(x, u_0(x)+\lambda u_1(x))$ is holomorphic in $\C$ with values in $C^{\alpha}(M)$ can be proved similarly, see  also \cite{Krup_Uhlmann_2}.  The holomorphy of $\lambda\mapsto A(x, u_0(x)+\lambda u_1(x))$ follows from the fact that in view of \eqref{eq_ap2_7}, the series 
\[
\sum_{k=0}^\infty \frac{A_k}{k!} (u_0+\lambda u_1)^k,
\]
converges in $C^{1,\alpha}(M, T^*M)$, locally uniformly in $\lambda\in \C$. 

We have $F(0,0)=0$ and the partial differential $\p_u F(0,0):B_2\to B_3$ is given by 
\[
\p_uF(0,0)v=(d^*_{\overline{A_0}}d_{A_0} v+V_1v, v|_{\p M}). 
\]
By the assumption (i), we have that the map $\p_u F(0,0):B_2\to B_3$ is a linear isomorphism, see \cite[Theorem 6.15]{Gil_Tru_book}. 

An application of the implicit function theorem, see \cite[p. 144]{Poschel_Trub_book}, allows us to conclude that there exists  $\delta>0$ and a unique holomorphic map $S: B_\delta(\p M)\to C^{2,\alpha}(M)$ such that $S(0)=0$ and $F(f, S(f))=0$ for all $f\in B_\delta(\p M)$. Setting $u=S(f)$ and noting that $S$ is Lipschitz continuous with $S(0)=0$, we see that
\[
\|u\|_{C^{2,\alpha}(M)}\le C\|f\|_{C^{2,\alpha}(\p M)}.
\]
The proof is complete.
\end{proof}

\section{First order boundary determination of potentials}

\label{app_boundary_determination}

When proving Theorem \ref{thm_main} and Proposition \ref{prop_density_form}, an important step consists in determining 
the boundary values, as well as the normal derivatives, of a scalar function and a $1$--form, via suitable orthogonality relations involving harmonic functions on the manifold. The purpose of this section is to carry our this step. In doing so, we shall rely on the methods developed in \cite{Brown_2001}, \cite{Brown_Salo_2006}, with suitable modifications in \cite[Appendix]{Guill_Tzou_2011}, where the boundary values of a scalar potential and a vector field are recovered. The main contribution of this section is that we push the methods a little further, in order to recover the first order normal derivatives of the potential and the $1$--form under limited regularity assumptions,  see also \cite{Aless_deHoop_Gaburro_Sincich_2018}.  We would like to mention the works 
\cite{Brown_Salo_2006}, \cite[Appendix]{Garcia_Zhang_2016}, where the gradient of a $C^1$--conductivity at the boundary of a Euclidean domain is recovered, see also \cite{Alessandrini_1990},  \cite{Caro_Andoni_2017}, \cite{Caro_Merono}. We refer to  \cite{Kohn_Vogelius_1984}, \cite{Lee_Uhlmann}, \cite{Nakamura_Sun_Uhlmann}, \cite{Sylvester_Uhlmann_1988}, where the entire Taylor series at the boundary of $C^\infty$--coefficients are recovered.

To proceed, we shall need the following density result for space of $L^2$ harmonic functions, see also \cite[Corollary 2.14]{Choe_Koo_yi_2004}  for a different approach in the Euclidean setting. 
\begin{prop}
\label{prop_density} 
Let $(M,g)$ be smooth compact Riemannian manifold of dimension $n\ge 2$ with smooth boundary. 
The set of harmonic functions on $M^\text{int}$ that are smooth up to the boundary is dense in the space of $L^2$-harmonic functions, in the $L^2$--topology. 
\end{prop}

\begin{proof}
Let $u\in L^2(M)$ be harmonic, i.e. $-\Delta_g u=0$ in $M^{\text{int}}$. Then by the partial hypoellipticity of the Laplacian, see \cite[Theorem 26.1]{Eskin_book}, we have $f=u|_{\p M}\in H^{-\frac{1}{2}}(\p M)$.  There exists therefore a sequence $f_j\in C^\infty(\p M)$, $j=1,2,\dots$, such that $\|f_j-f\|_{H^{-\frac{1}{2}}(\p M)}\to 0$ as $j\to\infty$.  The Dirichlet problem
\[
\begin{cases}
-\Delta_{g}u_j=0 & \text{in}\quad M^{\text{int}},\\
u_j|_{\p M}=f_j,
\end{cases}
\]
has a unique solution $u_j\in H^1(M)$, and by the boundary elliptic regularity, $u_j\in C^\infty(M)$.  By \cite[Theorem 26.3]{Eskin_book}, we get  
\[
\|u_j-u\|_{L^2(M)}\le C\|f_j-f\|_{H^{-\frac{1}{2}}(\p M)}\to 0,
\]
as $j\to \infty$, establishing the proposition. 
\end{proof}

Our first boundary determination result is as follows. 
\begin{prop}
\label{prop_boundary_V}
Let $(M,g)$ be a conformally transversally anisotropic manifold of dimension $n\ge 3$, and  let $V\in C^{1,1}(M)$.   If
\begin{equation}
\label{eq_boun_1}
\int_M V u_1 u_2 dV_g=0,
\end{equation}
for all harmonic functions $u_1, u_2\in C^\infty(M)$,   then $V|_{\p M}=0$ and $\p_\nu V|_{\p M}=0$.
\end{prop}

\begin{proof}

First, by Proposition \ref{prop_density}, we see that \eqref{eq_boun_1} continues to hold for all harmonic functions $u_1,u_2\in L^2(M)$. To proceed, we shall follow \cite{Brown_2001}, \cite{Brown_Salo_2006}, constructing a family of functions, whose boundary values have a highly oscillatory behavior while becoming increasingly concentrated near a given point on the boundary of $M$.  To convert such functions to harmonic functions, we follow the idea of \cite[Appendix]{Guill_Tzou_2011} and rely on a Carleman estimate for the conjugated Laplacian with a gain of two derivatives, established in \cite[Lemma 2.1]{Salo_Tzou_2009} in the Euclidean case and in \cite[Proposition 2.2]{Krup_Uhlmann_magn_2018} in the conformally transversally anisotropic case.

Let $x_0\in \p M$ and let $(x_1,\dots, x_n)$ be the boundary normal coordinates centered at $x_0$ so that in these coordinates, $x_0 =0$, the boundary $\p M$ is given by $\{x_n=0\}$, and $M^{\text{int}}$ is given by $\{x_n > 0\}$. We have, see \cite{Lee_Uhlmann},   
\begin{equation}
\label{eq_boun_2_metric}
g(x',x_n)=\sum_{\alpha,\beta=1}^{n-1}g_{\alpha\beta}(x)dx_\alpha dx_\beta+(dx_n)^2,
\end{equation}
and we may also assume that the coordinates $x' = (x_1, \dots, x_{n-1})$ are chosen so that 
\begin{equation}
\label{eq_boun_2}
g^{\alpha\beta}(x',0)=\delta^{\alpha\beta}+\mathcal{O}(|x'|^2), \quad 1\le \alpha,\beta\le n-1,
\end{equation}
see \cite[Chapter 2, Section 8, p. 56]{Petersen_book}. 
Therefore, 
\begin{equation}
\label{eq_boun_17}
g^{\alpha\beta}(x',x_n)=g^{\alpha\beta}(x',0)+\mathcal{O}(x_n)=\delta^{\alpha\beta}+\mathcal{O}(|x'|^2)+\mathcal{O}(x_n). 
\end{equation}

In view of \eqref{eq_boun_2}, we have
\begin{equation}
\label{eq_boun_2_lap}
-\Delta_g=D_{x_n}^2+\sum_{\alpha,\beta=1}^{n-1}g^{\alpha\beta}(x)D_{x_\alpha}D_{x_\beta}+f(x)D_{x_n}+R(x,D_{x'}),
\end{equation}
where 
$f$ is a smooth function and $R$ is a differential operator of order $1$ in $x'$ with smooth coefficients, see  \cite{Lee_Uhlmann}.  Notice that in the local coordinates, $T_{x_0}\p M=\R^{n-1}$, equipped with the Euclidean metric. 
The unit tangent vector $\tau$ is then given by $\tau=(\tau',0)$ where $\tau'\in \R^{n-1}$, $|\tau'|=1$.   Associated to the tangent vector $\tau'$ is the covector $\xi'_\alpha=\sum_{\beta=1}^{n-1} g_{\alpha \beta}(0) \tau'_\beta=\tau'_\alpha\in T^*_{x_0}\p M$. 

Let $\eta\in C^\infty_0(\R^n;\R)$ be such that $\supp(\eta)$ is in a small neighborhood of $0$, and 
\begin{equation}
\label{eq_boun_4}
\int_{\R^{n-1}}\eta(x',0)^2dx'=1.
\end{equation}
Let $\frac{1}{3}\le \alpha\le \frac{1}{2}$. Following \cite{Brown_Salo_2006}, in the boundary normal coordinates,  we set 
\begin{equation}
\label{eq_boun_5}
v_0(x)=\eta\bigg(\frac{x}{\lambda^{\alpha}}\bigg)e^{\frac{i}{\lambda}(\tau'\cdot x'+ ix_n)}, \quad 0<\lambda\ll 1,
\end{equation}
so that  $v_0\in C^\infty(M)$, with $\supp(v_0)$ in  $\mathcal{O}(\lambda^{\alpha})$ neighborhood of $x_0=0$. Here $\tau'$ is viewed as a covector.   

A direct computation 
\begin{equation}
\label{eq_boun_10}
\begin{aligned}
\|v_0\|_{L^2(M)}^2=\mathcal{O}(1)\int_{|x|\le c\lambda^{\alpha},x_n\ge 0} e^{-\frac{2x_n}{\lambda}}dx'dx_n=\mathcal{O}(\lambda^{\alpha(n-1)})\int_0^\infty e^{-2t}\lambda dt \\
=\mathcal{O}(\lambda^{\alpha(n-1)+1}),
\end{aligned}
\end{equation}
as $\lambda\to 0$, shows that 
\begin{equation}
\label{eq_boun_6}
\|v_0\|_{L^2(M)}=\mathcal{O}(\lambda^{\frac{\alpha(n-1)}{2}+\frac{1}{2}}). 
\end{equation}

Following \cite[Appendix]{Guill_Tzou_2011}, we shall construct a harmonic function $u\in L^2(M)$ of the form 
\[
u=v_0+r,
\] 
and therefore, we need to find $r\in L^2(M)$ satisfying 
\begin{equation}
\label{eq_boun_6_remainder}
\Delta_g r=-\Delta_g v_0\quad \text{in}\quad M^{\text{int}}.
\end{equation}
To that end, we shall rely on the following Carleman estimate for the conjugated Laplacian with a gain of two derivatives established in \cite[Lemma 2.1]{Salo_Tzou_2009}, \cite[Proposition 2.2]{Krup_Uhlmann_magn_2018}: for all $0<h\ll 1$ and all $v\in C^\infty_0(M^{\text{int}})$, we have 
\begin{equation}
\label{eq_boun_6_Carleman}
\|v\|_{H^2_{\text{scl}}(M^{\text{int}})}\le \frac{C}{h} \| e^{\frac{\varphi}{h}}\circ (-h^2\Delta_g)\circ e^{-\frac{\varphi}{h}} v \|_{L^2(M)}.
\end{equation}
Here the limiting Carleman weight $\varphi(x)=x_1$. Using a standard argument, one can convert the Carleman estimate \eqref{eq_boun_6_Carleman} into a solvability result. Applying this solvability result with $h>0$ small but fixed, we see that there exists a solution $r\in L^2(M)$ of \eqref{eq_boun_6_remainder} such that 
\begin{equation}
\label{eq_boun_6_remainder_bound}
\|r\|_{L^2(M)}\le C\|\Delta_g v_0\|_{H^{-2}(M^{\text{int}})}. 
\end{equation}

Next we claim that 
\begin{equation}
\label{eq_boun_7}
\|\Delta_g v_0\|_{H^{-2}(M^{\text{int}})}=\mathcal{O}(\lambda^{\frac{\alpha(n-3)}{2}+\frac{3}{2}}), \quad \frac{1}{3}\le \alpha\le \frac{1}{2},
\end{equation}
as $\lambda\to 0$. In order to prove \eqref{eq_boun_7}, we first compute the Euclidean Laplacian acting on $v_0$, 
\begin{equation}
\label{eq_boun_11}
\begin{aligned}
\Delta v_0=& e^{\frac{i}{\lambda}(\tau'\cdot x'+ix_n)}\bigg[  \lambda^{-2\alpha} (\Delta \eta)\bigg(\frac{x}{\lambda^{\alpha}}\bigg)+ 2i \lambda^{-\alpha -1} (\nabla \eta)\bigg(\frac{x}{\lambda^{\alpha}}\bigg)\cdot(\tau',i)\\
& -
\lambda^{-2} (\tau',i)\cdot(\tau',i)\eta\bigg(\frac{x}{\lambda^{\alpha}}\bigg)
\bigg]\\
= & e^{\frac{i}{\lambda}(\tau'\cdot x'+ix_n)}\bigg[  \lambda^{-2\alpha} (\Delta \eta)\bigg(\frac{x}{\lambda^{\alpha}}\bigg)+ 2i \lambda^{-\alpha -1} (\nabla \eta)\bigg(\frac{x}{\lambda^{\alpha}}\bigg)\cdot(\tau',i)\bigg],
\end{aligned}
\end{equation}
where we have used that $(\tau',i)\cdot(\tau',i)=0$.  The second term in the right hand side of \eqref{eq_boun_11} has the worst growth as $\alpha\to 0$ and we will analyze it. The first term in the right hand side of \eqref{eq_boun_11}  can be treated in a similar fashion.  To that end, we note that the second term in the right hand side of \eqref{eq_boun_11} has the form 
\[
\lambda^{-\alpha -1} \chi\bigg(\frac{x}{\lambda^{\alpha}}\bigg)e^{\frac{i}{\lambda}(\tau'\cdot x'+ix_n)},
\]
where $\chi\in C^\infty(\R^n)$ is supported in a small neighborhood of $0$, and we can proceed similarly to \cite[Appendix]{Guill_Tzou_2011}. Setting 
\[
L=\frac{\nabla \overline{\phi}\cdot \nabla}{i |\nabla \phi |^2}=\frac{1}{2i}\nabla \overline{\phi}\cdot \nabla, \quad \phi=\tau'\cdot x'+ix_n,
\]
we get $Le^{\frac{i}{\lambda}(\tau'\cdot x'+ix_n)}=\lambda^{-1}e^{\frac{i}{\lambda}(\tau'\cdot x'+ix_n)}$.  Letting $\psi\in C^\infty_0(M^{\text{int}})$, and integrating by parts twice using the operator $L$, we obtain that 
\begin{equation}
\label{eq_boun_12}
\begin{aligned}
\lambda^{-\alpha -1} \int_M &\chi\bigg(\frac{x}{\lambda^{\alpha}}\bigg)\psi (x) e^{\frac{i}{\lambda}(\tau'\cdot x'+ix_n)}dV_g
\\
&=\lambda^{-\alpha -1} \lambda^2 \int_M  (L)^2\bigg(\chi\bigg(\frac{x}{\lambda^{\alpha}}\bigg)\psi (x) |g(x)|^{1/2}\bigg) e^{\frac{i}{\lambda}(\tau'\cdot x'+ix_n)}dx,
\end{aligned}
\end{equation}
since the transpose $L^t= -L$.  The term in the right hand side of \eqref{eq_boun_12}, where the bound cannot be improved integrating by parts further, will occur when the operator $(L)^2$ falls on $\psi$, and in this case, using the Cauchy--Schwarz inequality and a computation similar to \eqref{eq_boun_10},  we get 
\begin{equation}
\label{eq_boun_13}
\begin{aligned}
\bigg| & \lambda^{-\alpha +1}  \int_M  \chi\bigg(\frac{x}{\lambda^{\alpha}} \bigg) e^{\frac{i}{\lambda}(\tau'\cdot x'+ix_n)} (L)^2(\psi (x)) dV_g \bigg|\\
&\le \lambda^{-\alpha +1}\bigg\|\chi\bigg(\frac{x}{\lambda^{\alpha}} \bigg) e^{\frac{i}{\lambda}(\tau'\cdot x'+ix_n)}\bigg\|_{L^2(M)}\|\psi\|_{H^2(M^{\text{int}})}\le \mathcal{O}(\lambda^{\frac{\alpha(n-3)}{2}+\frac{3}{2}})\|\psi\|_{H^2(M^{\text{int}})}.
\end{aligned}
\end{equation}
Proceeding similarly, integrating by parts using the operator $L$, if needed, we can bound all the other terms in \eqref{eq_boun_12} with the same bound as in \eqref{eq_boun_13}. Therefore, it follows from \eqref{eq_boun_11} and \eqref{eq_boun_13} that for $0<\alpha\le 1/2$, we have 
\begin{equation}
\label{eq_boun_14}
\|\Delta v_0\|_{H^{-2}(M^{\text{int}})}=\mathcal{O}(\lambda^{\frac{\alpha(n-3)}{2}+\frac{3}{2}}), 
\end{equation}
as $\lambda\to 0$. To get the bound \eqref{eq_boun_7} for the Laplace-Beltrami operator, we notice that in view of \eqref{eq_boun_2},  \eqref{eq_boun_2_lap}, and \eqref{eq_boun_14}, we have to bound 
\begin{equation}
\label{eq_boun_15}
\sum_{\alpha,\beta=1}^{n-1}(g^{\alpha\beta}(x)-\delta^{\alpha\beta})D_{x_\alpha}D_{x_\beta}v_0+f(x)D_{x_n}v_0+R(x,D_{x'})v_0
\end{equation}
in $H^{-2}(M^{\text{int}})$. Let us proceed to bound the first term. To that end, we  compute
\begin{equation}
\label{eq_boun_16}
\begin{aligned}
 D_{x_\alpha}D_{x_\beta}v_0= & e^{\frac{i}{\lambda}(\tau'\cdot x'+ix_n)}\bigg[  \lambda^{-2\alpha} (D_{x_\alpha}D_{x_\beta} \eta)\bigg(\frac{x}{\lambda^{\alpha}}\bigg)+  \lambda^{-1- \alpha} (D_{x_\alpha} \eta)\bigg(\frac{x}{\lambda^{\alpha}}\bigg)\tau_\beta\\
& +  \lambda^{-1-\alpha} (D_{x_\beta} \eta)\bigg(\frac{x}{\lambda^{\alpha}}\bigg)\tau_\alpha + \lambda^{-2} \tau_\alpha \tau_\beta \eta\bigg(\frac{x}{\lambda^{\alpha}}\bigg)
\bigg].
\end{aligned}
\end{equation}

The worst growth as $\lambda\to 0$ is in the fourth term in \eqref{eq_boun_16}, and therefore, in view of \eqref{eq_boun_15}, we proceed to bound
\[
\lambda^{-2} (g^{\alpha\beta}-\delta^{\alpha\beta})\chi\bigg(\frac{x}{\lambda^{\alpha}}\bigg)e^{\frac{i}{\lambda}(\tau'\cdot x'+ix_n)}, \quad  \chi (x) =\tau_\alpha \tau_\beta \eta(x),
\]
in $H^{-2}(M^{\text{int}})$. The other terms in the first term in \eqref{eq_boun_15} can be bounded similarly.  As before, integrating by parts twice using the operator $L$, we get 
\begin{equation}
\label{eq_boun_17_1}
\begin{aligned}
\lambda^{-2} \int_M &(g^{\alpha\beta}-\delta^{\alpha\beta})\chi\bigg(\frac{x}{\lambda^{\alpha}}\bigg)e^{\frac{i}{\lambda}(\tau'\cdot x'+ix_n)}\psi dV_g\\
&= \int_M (L)^2\bigg( (g^{\alpha\beta}-\delta^{\alpha\beta})\chi\bigg(\frac{x}{\lambda^{\alpha}}\bigg) \psi |g(x)|^{1/2}\bigg) e^{\frac{i}{\lambda}(\tau'\cdot x'+ix_n)}dx.
\end{aligned}
\end{equation}
The term in the right hand side of \eqref{eq_boun_17_1} where the bound cannot be improved occurs when the operator $(L)^2$ falls on $\psi$, and in this case, using the Cauchy--Schwarz inequality, \eqref{eq_boun_17},   and a computation similar to \eqref{eq_boun_10},  we get  
\begin{equation}
\label{eq_boun_18}
\begin{aligned}
\bigg| \int_M &(g^{\alpha\beta}-\delta^{\alpha\beta})\chi\bigg(\frac{x}{\lambda^{\alpha}}\bigg)  e^{\frac{i}{\lambda}(\tau'\cdot x'+ix_n)} (L)^2\psi dV_g \bigg|\\
&\le \bigg(\int_M (\mathcal{O}(|x'|^4)+\mathcal{O}(x_n^2)) \chi^2\bigg(\frac{x}{\lambda^{\alpha}}\bigg)e^{-\frac{2x_n}{\lambda}}dV_g\bigg)^{1/2}\|\psi\|_{H^2(M^{\text{int}})}\\
&\le \bigg(\mathcal{O}(\lambda^{2\alpha} \lambda^{\frac{\alpha(n-1)}{2}+\frac{1}{2}})  +\mathcal{O}(\lambda^{\frac{\alpha(n-1)}{2}}) \bigg(\int_0^\infty x_n^2 e^{-\frac{2x_n}{\lambda}}dx_n\bigg)^{1/2}\bigg)\|\psi\|_{H^2(M^{\text{int}})}\\=&\bigg(\mathcal{O}(\lambda^{\frac{\alpha(n+3)}{2}+\frac{1}{2}})+ \mathcal{O}(\lambda^{\frac{\alpha(n-1)}{2}+\frac{3}{2}})\bigg) \|\psi\|_{H^2(M^{\text{int}})}.
\end{aligned}
\end{equation} 
The growth in $\lambda$ in  \eqref{eq_boun_18} is smaller than or equal to that in  the desired bound \eqref{eq_boun_7} provided that $\alpha\ge 1/3$.  Proceeding similarly, integrating by parts, using the operator $L$ if needed, we can bound all the other terms in  \eqref{eq_boun_17_1} by the bound which is the same or better than
\[
\mathcal{O}(\lambda^{\frac{\alpha(n-3)}{2}+\frac{3}{2}})\|\psi\|_{H^2(M^{\text{int}})}.
\]
Thus, using this and in view of \eqref{eq_boun_15},  \eqref{eq_boun_16},  \eqref{eq_boun_17_1}, \eqref{eq_boun_18}, we conclude that 
\begin{equation}
\label{eq_boun_19}
\|\sum_{\alpha,\beta=1}^{n-1}(g^{\alpha\beta}(x)-\delta^{\alpha\beta})D_{x_\alpha}D_{x_\beta}v_0 \|_{H^{-2}(M^{\text{int}})}= \mathcal{O}(\lambda^{\frac{\alpha(n-3)}{2}+\frac{3}{2}}),
\end{equation}
provided that $\frac{1}{3}\le \alpha\le \frac{1}{2}$.  Finally, as $R(x,D_{x'})$ is a differential operator of order $1$ in $x'$, similarly, we get  
\begin{equation}
\label{eq_boun_20}
\|f(x)D_{x_n}v_0+R(x,D_{x'})v_0\|_{H^{-2}(M^{\text{int}})}=\mathcal{O}(\lambda^{\frac{\alpha(n-1)}{2}+\frac{3}{2}}),
\end{equation}
which is better than the desired bound \eqref{eq_boun_7}.  Hence, combining \eqref{eq_boun_14}, \eqref{eq_boun_19}, and \eqref{eq_boun_20}, we get \eqref{eq_boun_7}. 

Now it follows from \eqref{eq_boun_6_remainder_bound} and \eqref{eq_boun_7} that 
\begin{equation}
\label{eq_boun_21}
\|r\|_{L^2(M)}=\mathcal{O}(\lambda^{\frac{\alpha(n-3)}{2}+\frac{3}{2}}), \quad \frac{1}{3}\le \alpha\le \frac{1}{2},
\end{equation}
as $\lambda\to 0$.  Notice that the bound for $r$ in $L^2$ is better than the bound for $v_0$ in $L^2$, cf. \eqref{eq_boun_6}. 

Letting 
\begin{equation}
\label{eq_boun_21_u_1}
u_1=v_0+r, \quad u_2=\overline{v_0+r},
\end{equation}
in \eqref{eq_boun_1}, and multiplying \eqref{eq_boun_1} by $\lambda^{-\alpha(n-1)-1}$,  we get 
\begin{equation}
\label{eq_boun_22}
0=\lambda^{-\alpha(n-1)-1}\int_M V(v_0+r)(\overline{v_0}+\overline{r})dV_g=\lambda^{-\alpha(n-1)-1} (I_1+I_2+I_3).
\end{equation}
Here
\[
I_1=\int_M V |v_0|^2 dV_g, \quad  I_2=\int_M V (v_0 \overline{r}+\overline{v_0}r) dV_g, \quad I_3=\int_M V |r|^2 dV_g.
\]
Using \eqref{eq_boun_6} and \eqref{eq_boun_21}, we obtain that 
\begin{equation}
\label{eq_boun_23}
\lambda^{-\alpha(n-1)-1} |I_2|\le \mathcal{O}(\lambda^{-\alpha(n-1)-1} )\|v_0\|_{L^2(M)}\|r\|_{L^2(M)}=\mathcal{O}(\lambda^{1-\alpha}),
\end{equation}
and 
\begin{equation}
\label{eq_boun_24}
\lambda^{-\alpha(n-1)-1} |I_3|\le \mathcal{O}(\lambda^{-\alpha(n-1)-1} )\|r\|_{L^2(M)}^2=\mathcal{O}(\lambda^{2-2\alpha}),
\end{equation}
as $\lambda\to 0$. 
Using \eqref{eq_boun_5}, \eqref{eq_boun_4}, the fact that $V$ is continuous, and making  the change of variables $y'=\frac{x'}{\lambda^{\alpha}}$, $y_n=\frac{x_n}{\lambda}$, we get 
\begin{equation}
\label{eq_boun_25}
\begin{aligned}
\lim_{\lambda\to 0}& \lambda^{-\alpha(n-1)-1}I_1\\
&= \lim_{\lambda\to 0}\int_{\R^{n-1}}\int_0^\infty V(\lambda^{\alpha}y',\lambda y_n)\eta^2(y', \lambda^{1-\alpha}y_n)e^{-2y_n} |g(\lambda^{\alpha}y', \lambda y_n)|^{1/2}dy'dy_n\\
&=V(0)|g(0)|^{1/2}\int_0^{+\infty} e^{-2y_n}dy_n=\frac{1}{2}V(0).
\end{aligned}
\end{equation}
Passing to the limit $\lambda\to 0$ in \eqref{eq_boun_22} and using \eqref{eq_boun_23}, \eqref{eq_boun_24}, and \eqref{eq_boun_25}, we obtain that $V(0)=0$, showing that $V|_{\p M}=0$. Notice that here we can consider any $\alpha$, $\frac{1}{3}\le \alpha\le \frac{1}{2}$. 

Next we would like to prove that $\p_\nu V|_{\p M}=0$. To that end, as before, we let $x_0\in \p M$ and consider boundary normal coordinates centered at $x_0$. As $V\in C^{1,1}$ and $V(x',0)=0$, using the fundamental theorem of calculus and integrating by parts, we have for $x$ near $x_0=0$, 
\begin{equation}
\label{eq_boun_26}
\begin{aligned}
V(x',x_n)=\int_0^1 \frac{d}{dt} V(x',tx_n)d(t-1)=V'_{x_n}(x',0)x_n+\int_0^1 (1-t) \frac{d^2}{dt^2} V(x',tx_n)\\
=V'_{x_n}(x',0)x_n+\int_0^1 (1-t)  V''_{x_nx_n}(x',tx_n)x_n^2dt=V'_{x_n}(x',0)x_n+ \mathcal{O}(x_n^2). 
\end{aligned}
\end{equation}
Now substituting $u_1$ and $u_2$ be given by \eqref{eq_boun_21_u_1} 
in \eqref{eq_boun_1}, multiplying \eqref{eq_boun_1} by $\lambda^{-\alpha(n-1)-2}$, and using \eqref{eq_boun_26}, we get 
\begin{equation}
\label{eq_boun_27}
0=\lambda^{-\alpha(n-1)-2}\int_M V(v_0+r)(\overline{v_0}+\overline{r})dV_g=\lambda^{-\alpha(n-1)-2} (I_{1,1} + I_{1,2}+I_2+I_3).
\end{equation}
Here 
\begin{align*}
&I_{1,1}=\int_M V'_{x_n}(x',0)x_n |v_0|^2 dV_g, \quad  I_{1,2}=\int_M \mathcal{O}(x_n^2) |v_0|^2 dV_g,\\
&I_2=\int_M V (v_0 \overline{r}+\overline{v_0}r) dV_g, \quad I_3=\int_M V |r|^2 dV_g.
\end{align*}

Using \eqref{eq_boun_5}, \eqref{eq_boun_4}, making  the change of variables $y'=\frac{x'}{\lambda^{\alpha}}$, $y_n=\frac{x_n}{\lambda}$, and using that $V'_{x_n}$ is continuous,  we obtain that 
\begin{equation}
\label{eq_boun_28}
\begin{aligned}
\lim_{\lambda\to 0}& \lambda^{-\alpha(n-1)-2}I_{1,1}\\
&=  \lim_{\lambda\to 0} \int_{\R^{n-1}}\int_0^\infty V'_{x_n}(\lambda^\alpha y',0) \eta^2(y', \lambda^{1-\alpha}y_n)y_ne^{-2y_n} |g(\lambda^{\alpha}y', \lambda y_n)|^{1/2}dy'dy_n\\
&=V'_{x_n}(0)|g(0)|^{1/2}\int_0^{+\infty} y_n e^{-2y_n}dy_n=\frac{1}{4}V'_{x_n}(0).
\end{aligned}
\end{equation}
Using \eqref{eq_boun_5}, we get 
\begin{equation}
\label{eq_boun_29}
 \lambda^{-\alpha(n-1)-2}|I_{1,2}|\le \mathcal{O}(\lambda^{-\alpha(n-1)-2})\int_{|x|\le c\lambda^{\alpha},x_n\ge 0}x_n^2e^{-\frac{2x_n}{\lambda}}dx'dx_n=\mathcal{O}(\lambda).
\end{equation}

Using \eqref{eq_boun_21}, we see that 
\begin{equation}
\label{eq_boun_30}
 \lambda^{-\alpha(n-1)-2}|I_{3}|\le \mathcal{O}(\lambda^{-\alpha(n-1)-2})\|r\|_{L^2(M)}^2=\mathcal{O}(\lambda^{1-2\alpha})=o(1),
\end{equation}
as $\lambda\to 0$, provided that $\alpha<\frac{1}{2}$. 

In view of \eqref{eq_boun_5} and \eqref{eq_boun_26}, we have
\[
\|Vv_0\|_{L^2(M)}=\bigg(\int_{|x|\le c\lambda^{\alpha},x_n\ge 0}\mathcal{O}(x_n^2)e^{-\frac{2x_n}{\lambda}}dx'dx_n\bigg)^{\frac{1}{2}}=\mathcal{O}(\lambda^{\frac{\alpha(n-1)}{2}+\frac{3}{2}}),
\]
and therefore, using  \eqref{eq_boun_21}, we obtain that 
\begin{equation}
\label{eq_boun_31}
 \lambda^{-\alpha(n-1)-2}|I_{2}|\le \mathcal{O}(\lambda^{-\alpha(n-1)-2})\|r\|_{L^2(M)}\|Vv_0\|_{L^2(M)}=\mathcal{O}(\lambda^{1-\alpha}).
\end{equation}
Passing to the limit $\lambda\to 0$ in \eqref{eq_boun_27}, and using \eqref{eq_boun_28}, \eqref{eq_boun_29}, \eqref{eq_boun_20}, and \eqref{eq_boun_31}, we get  $V'_{x_n}(0)=0$ provided that $\alpha$ is a fixed number satisfying $\frac{1}{3}\le \alpha<\frac{1}{2}$. This shows that $\p_{\nu}V|_{\p M}=0$. The proof is complete. 
\end{proof}

In order to prove Proposition \ref{prop_density_form}, we shall need the following boundary determination result.
\begin{prop}
\label{prop_boundary_A}
Let $(M,g)$ be a conformally transversally anisotropic manifold of  dimension $n\ge 3$. Let $A\in C^{1,1}(M, T^*M)$ be a $1$-form. If
\begin{equation}
\label{eq_boun_A_1}
\int_M\langle A, d u_1\rangle_g u_2 dV_g=0,
\end{equation}
for all harmonic functions $u_1, u_2\in C^\infty(M)$, then $A|_{\p M}=0$ and $\p_\nu A|_{\p M}=0$.
\end{prop}

\begin{proof}
First by Proposition \ref{prop_density}, we see that \eqref{eq_boun_A_1} holds for all harmonic functions $u_2\in L^2(M)$. 
To prove this result, we shall test the integral identity \eqref{eq_boun_A_1} with harmonic functions $u_2\in L^2(M)$, constructed in Proposition \ref{prop_boundary_V}, 
\begin{equation}
\label{eq_boun_32}
u_2=\overline{v_0+r}.
\end{equation}
Since for $u_1$ we need estimates in $H^1(M^\text{int})$, we shall construct $u_1$ following \cite{Brown_2001},  \cite{Brown_Salo_2006}, see also \cite[Appendix A]{Krup_Uhlmann_magn_2018}.  We let 
\begin{equation}
\label{eq_boun_33}
u_1=v_0+r_1,
\end{equation}
where $r_1\in H^1_0(M^{\text{int}})$ is a solution to the Dirichlet problem, 
\begin{equation}
\label{eq_boun_34}
\begin{cases}
-\Delta_g r_1=\Delta_g v_0 & \text{in}\quad M,\\
r_1|_{\p M}=0. 
\end{cases}
\end{equation}
Note that by boundary elliptic regularity, $r_1\in C^\infty(M)$, and therefore, $u_1\in C^\infty(M)$.

Applying the Lax--Milgram lemma to \eqref{eq_boun_34}, we get 
\begin{equation}
\label{eq_boun_35}
\|r_1\|_{H^1_0(M^{\text{int}})}\le C\|\Delta_g v_0\|_{H^{-1}(M^{\text{int}})}.
\end{equation}
Similarly to the bound \eqref{eq_boun_7}, one can show that 
\[
\|\Delta_g v_0\|_{H^{-1}(M^{\text{int}})}=\mathcal{O}(\lambda^{\frac{\alpha(n-3)}{2}+\frac{1}{2}}), \quad \frac{1}{3}\le \alpha\le \frac{1}{2},
\]
see also \cite[Appendix A]{Krup_Uhlmann_magn_2018}.
This bound together with \eqref{eq_boun_35} implies that 
\begin{equation}
\label{eq_boun_36}
\|r_1\|_{H^1(M^{\text{int}})}=\mathcal{O}(\lambda^{\frac{\alpha(n-3)}{2}+\frac{1}{2}}), \quad \frac{1}{3}\le \alpha\le \frac{1}{2},
\end{equation}
as $\lambda\to 0$. 

We shall also need the following bound
\begin{equation}
\label{eq_boun_36_1}
\|dv_0\|_{L^2(M)}=\mathcal{O}(\lambda^{\frac{\alpha(n-1)}{2}-\frac{1}{2}}),
\end{equation}
as $\lambda\to 0$, which is in view of \eqref{eq_boun_5} implied by the following estimate
\[
\|dv_0\|_{L^2(M)}\le \mathcal{O}(1)\bigg(\int_{|x|\le c\lambda^\alpha,x_n\ge 0}\lambda^{-2}e^{-\frac{2x_n}{\lambda}}dx'dx_n\bigg)^{1/2}=\mathcal{O}(\lambda^{\frac{\alpha(n-1)}{2}-\frac{1}{2}}).
\]

Now substituting $u_1$, $u_2$ given by \eqref{eq_boun_33}, \eqref{eq_boun_32}, respectively, into \eqref{eq_boun_A_1}, and multiplying \eqref{eq_boun_A_1} by $\lambda^{-\alpha(n-1)}$ , we get
\begin{equation}
\label{eq_boun_37}
0= \lambda^{-\alpha(n-1)}\int_{M} \langle A, dv_0+dr_1\rangle_g(\overline{v_0}+\overline{r})dV_g= \lambda^{-\alpha(n-1)}( I_1+I_2+I_3),
\end{equation}
where 
\begin{align*}
I_1=\int_{M} \langle A, dv_0\rangle_g \overline{v_0}dV_g, \quad I_2=\int_{M} \langle A, dr_1\rangle_g(\overline{v_0}+\overline{r})dV_g,\quad I_3= \int_{M} \langle A, dv_0\rangle_g \overline{r}dV_g.
\end{align*}

First using \eqref{eq_boun_5}, we write
\[
I_1=I_{1,1}+I_{1,2}, 
\]
where 
\begin{align*}
&I_{1,1}=i\lambda^{-1}  \int_M  \langle A, \tau'\cdot dx'+idx_n\rangle_g \eta^2\bigg(\frac{x}{\lambda^\alpha}\bigg)e^{-\frac{2x_n}{\lambda}}dV_g,\\
&I_{1,2}=\lambda^{-\alpha}  \int_M\bigg\langle A, (d\eta)\bigg(\frac{x}{\lambda^\alpha}\bigg) \bigg\rangle_g \eta\bigg(\frac{x}{\lambda^\alpha}\bigg)e^{-\frac{2x_n}{\lambda}}dV_g.
\end{align*}
Using \eqref{eq_boun_2_metric}, and making  the change of variables $y'=\frac{x'}{\lambda^{\alpha}}$, $y_n=\frac{x_n}{\lambda}$, we get 
\begin{equation}
\label{eq_boun_38}
\begin{aligned}
& \lim_{\lambda\to 0}\lambda^{-\alpha(n-1)} I_{1,1}
 =i  \lim_{\lambda\to 0} \int_{\R^{n-1}}\int_0^{+\infty} |g(\lambda^\alpha y', \lambda y_n)|^{1/2}\eta^2(y', \lambda^{1-\alpha}y_n)e^{-2 y_n}\\
&  \bigg(\sum_{\alpha,\beta=1}^{n-1} g^{\alpha\beta}(\lambda^\alpha y', \lambda y_n)A_\alpha(\lambda^\alpha y', \lambda y_n)\tau'_\beta+ A_n(\lambda^\alpha y', \lambda y_n)i \bigg)
 dy'dy_n\\
 &=i \bigg(\sum_{\alpha,\beta=1}^{n-1} g^{\alpha\beta}(0)A_\alpha(0)\tau'_\beta+ A_n(0)i \bigg)|g(0)|^{1/2}\int_0^{+\infty}e^{-2 y_n}dy_n=\frac{i}{2}\langle A(0), (\tau',i)\rangle. 
 \end{aligned}
\end{equation}

Estimating similarly as in  \eqref{eq_boun_10}, we get 
\begin{equation}
\label{eq_boun_39}
\lambda^{-\alpha(n-1)} |I_{1,2}|\le \mathcal{O}(\lambda^{-\alpha n})\bigg\|(d\eta)\bigg(\frac{x}{\lambda^\alpha}\bigg)\bigg\|_{L^2(M)}\bigg\|\eta\bigg(\frac{x}{\lambda^\alpha}\bigg)e^{-\frac{2x_n}{\lambda}}\bigg\|_{L^2(M)}=\mathcal{O}(\lambda^{\frac{1-\alpha}{2}}).
\end{equation}

Using \eqref{eq_boun_6}, \eqref{eq_boun_21}, and \eqref{eq_boun_36}, we see that 
\begin{equation}
\label{eq_boun_40}
\lambda^{-\alpha(n-1)} |I_{2}|\le \mathcal{O}(\lambda^{-\alpha(n-1)} )\|dr_1\|_{L^2(M)}\|v_0+r\|_{L^2(M)}= \mathcal{O}(\lambda^{1-\alpha}).
\end{equation}

Finally, using \eqref{eq_boun_36_1} and \eqref{eq_boun_21}, we obtain that 
\begin{equation}
\label{eq_boun_41}
\lambda^{-\alpha(n-1)} |I_{3}|\le \mathcal{O}(\lambda^{-\alpha(n-1)} )\|dv_0\|_{L^2(M)}\|r\|_{L^2(M)}= \mathcal{O}(\lambda^{1-\alpha}).
\end{equation}

Passing to the limit $\lambda\to 0$ in \eqref{eq_boun_37} and using \eqref{eq_boun_38}, \eqref{eq_boun_39}, \eqref{eq_boun_40}, and \eqref{eq_boun_41}, we conclude that $\langle A(0), (\tau',i)\rangle=0$. Now changing  $\tau'$ to $-\tau'$, we see that $A_n(0)=0$, and therefore, $\langle A'(0), \tau'\rangle=0$, where $A'=(A_1,\dots, A_{n-1})$. As $\tau'\in \R^{n-1}$ is an arbitrary tangent vector to $\p M$ at $x_0=0$, we get $A'(0)=0$. This shows that $A|_{\p M}=0$. 

Next we shall show that $\p_\nu A|_{\p M}=0$. To that end, as before, we let $x_0\in \p M$ and consider the boundary normal coordinates centered at $x_0$. Applying computations similar to \eqref{eq_boun_26} to each component of $A$, we get 
\begin{equation}
\label{eq_boun_41_1}
A(x',x_n)=(A'_{1 x_n}, \dots, A'_{n x_n})(x',0)x_n+\mathcal{O}(x_n^2)=\p_{x_n}A(x',0)x_n+ \mathcal{O}(x_n^2). 
\end{equation}

Substituting $u_1$ and $u_2$ given by \eqref{eq_boun_33} and \eqref{eq_boun_32} into \eqref{eq_boun_A_1}, and multiplying \eqref{eq_boun_A_1} by $\lambda^{-\alpha(n-1)-1}$, we have in view of \eqref{eq_boun_41_1},
\begin{equation}
\label{eq_boun_42}
0=\lambda^{-\alpha(n-1)-1}\int_M \langle A, dv_0+dr_1\rangle_g (\overline{v_0}+\overline{r})dV_g=\lambda^{-\alpha(n-1)-1}(I_{1,1}+I_{1,2}+I_2+I_3+I_4),
\end{equation}
where
\begin{align*}
&I_{1,1}=\int_M \langle \p_{x_n}A(x',0)x_n, dv_0\rangle_g \overline{v_0}dV_g, \
I_{1,2}=\int_M \langle \mathcal{O}(x_n^2), dv_0\rangle_g \overline{v_0}dV_g,\\
&I_2=\int_M \langle A, dr_1\rangle_g \overline{v_0}dV_g, \quad I_3=\int_M \langle A, dr_1\rangle_g \overline{r}dV_g, \quad
I_4=\int_M \langle A, dv_0\rangle_g \overline{r}dV_g.
\end{align*}
In view of \eqref{eq_boun_5} we write
\begin{align*}
&I_{1,11}=i\lambda^{-1} \int_M \langle \p_{x_n}A(x',0)x_n, \tau'\cdot dx'+idx_n\rangle_g  \eta^2\bigg( \frac{x}{\lambda^\alpha}\bigg) e^{-\frac{2x_n}{\lambda}}dV_g,\\
&I_{1,12}=\lambda^{-\alpha} \int_M \bigg\langle \p_{x_n}A(x',0)x_n, (d\eta)\bigg( \frac{x}{\lambda^\alpha}\bigg)\bigg\rangle_g  \eta\bigg( \frac{x}{\lambda^\alpha}\bigg) e^{-\frac{2x_n}{\lambda}}dV_g.
\end{align*}

Using \eqref{eq_boun_2_metric}, and making  the change of variables $y'=\frac{x'}{\lambda^{\alpha}}$, $y_n=\frac{x_n}{\lambda}$, we get 
\begin{equation}
\label{eq_boun_43}
\begin{aligned}
& \lim_{\lambda\to 0}\lambda^{-\alpha(n-1)-1} I_{1,11}
 =i  \lim_{\lambda\to 0} \int_{\R^{n-1}}\int_0^{+\infty} |g(\lambda^\alpha y', \lambda y_n)|^{1/2}y_n\eta^2(y', \lambda^{1-\alpha}y_n)e^{-2 y_n}\\
&  \bigg(\sum_{\alpha,\beta=1}^{n-1} g^{\alpha\beta}(\lambda^\alpha y', \lambda y_n)\p_{x_n}A_\alpha(\lambda^\alpha y', 0)\tau'_\beta+ \p_{x_n}A_n(\lambda^\alpha y', 0)i \bigg)
 dy'dy_n\\
 &=i \bigg(\sum_{\alpha,\beta=1}^{n-1} g^{\alpha\beta}(0)\p_{x_n}A_\alpha(0)\tau'_\beta+ \p_{x_n}A_n(0)i \bigg)|g(0)|^{1/2}\int_0^{+\infty}y_ne^{-2 y_n}dy_n\\
 &=\frac{i}{4}\langle \p_{x_n}A(0), (\tau',i)\rangle. 
 \end{aligned}
\end{equation}
Estimating similarly as in  \eqref{eq_boun_10}, we get 
\begin{equation}
\label{eq_boun_44}
\begin{aligned}
\lambda^{-\alpha(n-1)-1} |I_{1,12}|\le \mathcal{O}(\lambda^{-\alpha n-1})\bigg\|(d\eta)\bigg(\frac{x}{\lambda^\alpha}\bigg)\bigg\|_{L^2(M)}\bigg\|x_n\eta\bigg(\frac{x}{\lambda^\alpha}\bigg)e^{-\frac{2x_n}{\lambda}}\bigg\|_{L^2(M)}\\
=\mathcal{O}(\lambda^{\frac{1-\alpha}{2}}).
\end{aligned}
\end{equation}

Using \eqref{eq_boun_36_1} and estimating similarly as in  \eqref{eq_boun_10}, we obtain that 
\begin{equation}
\label{eq_boun_45}
\lambda^{-\alpha(n-1)-1} |I_{1,2}|\le \mathcal{O}(\lambda^{-\alpha(n-1)-1}) \|dv_0\|_{L^2(M)}\|x_n^2v_0\|_{L^2(M)}=\mathcal{O}(\lambda).
\end{equation}

Using \eqref{eq_boun_36}, \eqref{eq_boun_41_1}, we get 
\begin{equation}
\label{eq_boun_46}
\lambda^{-\alpha(n-1)-1} |I_{2}|\le \mathcal{O}(\lambda^{-\alpha(n-1)-1}) \|dr_1 \|_{L^2(M)}\|x_nv_0\|_{L^2(M)}=\mathcal{O}(\lambda^{1-\alpha}).
\end{equation}

Using \eqref{eq_boun_36} and \eqref{eq_boun_21}, we have
\begin{equation}
\label{eq_boun_47}
\lambda^{-\alpha(n-1)-1} |I_{3}|\le \mathcal{O}(\lambda^{-\alpha(n-1)-1}) \|dr_1\|_{L^2(M)}\|r\|_{L^2(M)}=\mathcal{O}(\lambda^{1-2\alpha})=o(1),
\end{equation}
as $\lambda\to 0$, provided that $\alpha<\frac{1}{2}$. 

Using \eqref{eq_boun_41_1},  \eqref{eq_boun_21},  and the fact  that 
\[
\|x_n dv_0\|_{L^2(M)}=\mathcal{O}(\lambda^{\frac{\alpha(n-1)}{2}+\frac{1}{2}}), 
\]
we obtain that 
\begin{equation}
\label{eq_boun_48}
\lambda^{-\alpha(n-1)-1} |I_{4}|\le \mathcal{O}(\lambda^{-\alpha(n-1)-1}) \|x_n dv_0\|_{L^2(M)}\|r\|_{L^2(M)}=\mathcal{O}(\lambda^{1-\alpha}).
\end{equation}

Let us fix $\frac{1}{3}\le \alpha<\frac{1}{2}$. Passing to the limit $\lambda\to 0$ in  \eqref{eq_boun_42} and using \eqref{eq_boun_43}, \eqref{eq_boun_44}, \eqref{eq_boun_45}, \eqref{eq_boun_46},   \eqref{eq_boun_47}, \eqref{eq_boun_48}, we conclude that  $\langle \p_{x_n}A(0), (\tau',i)\rangle=0$, and therefore, $\p_{x_n}A(0)=0$. This shows that $\p_\nu A|_{\p M}=0$. The proof is complete.  

\end{proof}

Finally, in order to prove Theorem \ref{thm_main} we shall need the following boundary determination result. 

\begin{prop}
\label{prop_boundary_A_V}
Let $(M,g)$ be a conformally transversally anisotropic manifold of dimension $n\ge 3$. Let $A\in C^{1,1}(M, T^*M)$ be a $1$-form and $V\in C^{1,1}(M)$. If
\begin{equation}
\label{eq_boun_A_V}
\int_M \big(4i \langle A, d(u_1u_2u_3)\rangle_g u_4-\big(3i d^*(A) +V\big) u_1u_2u_3u_4 \big)dV_g=0,
\end{equation}
for all harmonic functions $u_j\in C^{2,\alpha}(M)$, $j=1,\dots 4$, then $A|_{\p M}=0$ and $\p_\nu A|_{\p M}=0$.
\end{prop}

\begin{proof}
We also have 
\begin{equation}
\label{eq_boun_49}
\int_M \big(4i \langle A, d(u_2u_3u_4)\rangle_g u_1-\big(3i d^*(A) +V\big) u_1u_2u_3u_4 \big)dV_g=0.
\end{equation}

Subtracting \eqref{eq_boun_49} from \eqref{eq_boun_A_V}, we get 
\begin{equation}
\label{eq_boun_50}
\int_M \langle A, d(u_1u_2u_3)\rangle_g u_4 dV_g- \int_M  \langle A, d(u_2u_3u_4)\rangle_g u_1 dV_g=0.
\end{equation}
Letting $u_3=u_4=1$, \eqref{eq_boun_50} gives 
\begin{equation}
\label{eq_boun_51}
\int_M \langle A, d u_1\rangle_g u_2 dV_g=0,
\end{equation}
for all harmonic functions $u_1,u_2\in C^{2,\alpha}(M)$, and therefore for all  harmonic functions $u_1,u_2\in C^{\infty}(M)$. 
The result follows by an application of Proposition \ref{prop_boundary_A}. 
\end{proof}

When proving Proposition \ref{prop_density_form}, we shall also need the following standard density result. 
\begin{prop}
\label{prop_density_Holder} 
Let $(M,g)$ be smooth compact Riemannian manifold of dimension $n\ge 2$ with smooth boundary. 
The set of harmonic functions in $M^\text{int}$ that are smooth up to the boundary is dense in the space of $C^{2,\alpha}(M)$-harmonic functions, $0<\alpha<1$ , in the $C^{2,\beta}(M)$--topology, for $0<\beta<\alpha$. 
\end{prop}

\begin{proof}
The proof follows along the line of the proof of Proposition \ref{prop_density}. Indeed, let $u\in C^{2,\alpha}(M)$ be harmonic in $M^\text{int}$ and let $f=u|_{\p M}\in  C^{2,\alpha}(\p M)$. Let $0<\beta<\alpha$ and by density, there exists $f_j\in C^\infty(\p M)$ such that $\|f_j-f\|_{C^{2,\beta}(\p M)}\to 0$, as $j\to \infty$,  see \cite[Theorem A. 10]{Hormander_1976}. The Dirichlet problem,
\[
\begin{cases}-\Delta_g u_j=0& \text{in}\quad M^{\text{int}},\\
u_j|_{\p M}=f_j,
\end{cases}
\]
has a unique solution $u_j\in C^{2,\alpha}(M)$, and by elliptic regularity, $u_j\in C^\infty(M)$.  Using that $C^{2,\alpha}(M)\subset C^{2,\beta}(M)$, and the following bound for the solution to the Dirichlet problem for the Laplacian, see \cite[Section 6.3, p. 109]{Gil_Tru_book},
\[
\|u_j-u\|_{C^{2,\beta}(M)}\le C\|f_j-f\|_{C^{2,\beta}(\p M)}\to 0,
\]
we get the claim. 
\end{proof}

\section{Some facts about non-tangential geodesics}

\label{app_geodesics}

When proving Proposition \ref{prop_density_form}, in order to avoid the use of stationary and non-stationary phase arguments on the boundary of the manifold, we shall need the following result concerning non-tangential geodesics which was kindly proven for us by Gabriel Paternain.

\begin{prop}
\label{prop_geodesic} Let $(M_0, g_0)$ be a smooth compact Riemannian manifold of dimension $n\ge 2$ with  smooth boundary, and let $\gamma$ be a unit speed non-tangential geodesic on $M_0$  between boundary points. Then for each point $y_0=\gamma(t_0)\in M_0^{\text{int}}$, except finitely many, that exists a small neighborhood $W\subset S_{y_0}M_0=\{w\in T_{y_0}M_0: |w|_g=1\}$  of $w_0=\dot{\gamma}(t_0)$ such that for every $w\in W$, $w \neq w_0$,  the unit speed geodesic $\eta$ on $M_0$ passing through  $(y_0, w)$  is also non-tangential between boundary points, and $\gamma$ and $\eta$ do not  intersect each other at the boundary of $M_0$. Furthermore,  $\gamma$ and $\eta$ are distinct and are not reverses of each other.
\end{prop}

\begin{proof}
Let us first notice that the property of a geodesic being non-tangential is stable under small perturbations of the initial conditions, in view of the $C^{\infty}$--dependence of the geodesic flow on the initial conditions. Let $y_0=\gamma(t_0)\in M_0^{\text{int}}$. Reparametrizing the geodesic $\gamma$ if necessary, we may assume that $\gamma:[-S_1,S_2]\to M_0$, $0<S_1,S_2<\infty$, is such that $\gamma(0)=y_0$ and $\dot{\gamma}(0)=w_0$.  Let us consider the map 
\begin{equation}
\label{eq_0_D_1}
F_{y_0}:\text{neigh}(w_0, S_{y_0}M_0)\to \text{neigh}(\gamma(S_2), \p M_0), \quad F_{y_0}(w)=\pi(\varphi_{\tau(y_0,w)}(y_0, w)),
\end{equation}
where $\tau(y_0,w)$ is the exit time of the geodesic $\gamma_{y_0,w}$ through $(y_0,w)$, $ \varphi_t: SM_0\to SM_0$, $t\in \R$, is the geodesic flow, given by 
\begin{equation}
\label{eq_0_D_2}
\varphi_t(y,w)=(\gamma_{y,w}(t), \dot{\gamma}_{y,w}(t)),
\end{equation}
and $ \pi:  SM_0 \to M_0$, $\pi(y,w)=y$ is the canonical projection. 

The exit time $\tau(y_0,w)$ depends smoothly on $w$, in view of the implicit function theorem and the fact that the geodesic $\gamma$ is non-tangential. The map $F_{y_0}$ is therefore smooth, and we have $F_{y_0}(w_0)=\gamma(S_2)$. 

Let us now compute the differential of $F_{y_0}$ at $w_0$ acting on  a vector $\eta\in T_{w_0}S_{y_0}M_0$. To that end, consider a curve $w:(-a,a)\to S_{y_0}M_0$ such that $w(0)=w_0$ and $\dot{w}(0)=\eta$, and by the chain rule, we get 
\begin{equation}
\label{eq_0_D_3}
\begin{aligned}
F'_{y_0}(w_0)\eta=\frac{d}{ds}\bigg|_{s=0} F_{y_0}(w(s))= \frac{d}{ds}\bigg|_{s=0} \pi(\varphi_{\tau(y_0,w(s))}(y_0, w(s)))\\
=d\pi(\varphi_{\tau(y_0,w_0)}(y_0, w_0))\bigg(\frac{d}{dt}\bigg|_{t=\tau(y_0,w_0)}\varphi_t(y_0,w_0)\frac{\p \tau}{\p w}(y_0,w_0)\cdot \eta \\
+ \frac{\p \varphi_{\tau(y_0,w_0)}}{\p w}(y_0,w_0)\eta  \bigg).
\end{aligned}
\end{equation}
To proceed, we recall some facts about the geometry of the tangent bundle following \cite{Paternain_book}. First letting 
\[
V(y,w)=\ker (d\pi (y,w))\subset T_{(y,w)}SM_0
\]
be the vertical fiber of $TSM_0$ at $(y,w)$, see \cite[Section 1.3.1]{Paternain_book} we have the splitting 
\[
T_{(y,w)}SM_0=H(y,w)\oplus V(y,w),
\]
where $H(y,w)$ is the horizontal fiber of $TSM_0$ at $(y,w)$, see \cite[Section 1.3, p. 13]{Paternain_book}.
Both $V(y,w)$ and $H(y,w)$ can be identified with $S_yM_0$, and for $\xi\in T_{(y,w)}SM_0$, we write $\xi=(\xi^h,\xi^v)$, where $\xi^h,\xi^v\in S_yM_0$ are the corresponding horizontal and vertical parts of $\xi$. Let  $X:SM_0\to TSM_0$ be the geodesic vector field given by 
\begin{equation}
\label{eq_0_D_4}
X(\varphi_t(y,w))=\frac{d}{dt}\varphi_t(y,w).
\end{equation}
It follows from  \cite[Section 1.3, p. 13]{Paternain_book} that
we have 
\begin{equation}
\label{eq_0_D_5}
X(y,w)=(w,0). 
\end{equation}
Now in view of the above splitting, we have $(0,\eta)\in V(y_0,w_0)$, and therefore, we get 
\begin{equation}
\label{eq_0_D_6}
\frac{\p \varphi_{\tau(y_0,w_0)}}{\p w}(y_0,w_0)\eta=d \varphi_{\tau(y_0,w_0)} (y_0,w_0)(0,\eta).
\end{equation}
Using that $\tau(y_0,w_0)=S_2$, and \eqref{eq_0_D_2}, \eqref{eq_0_D_4}, \eqref{eq_0_D_5}, \eqref{eq_0_D_6}, we obtain from \eqref{eq_0_D_3} that 
\begin{equation}
\label{eq_0_D_7}
\begin{aligned}
F'_{y_0}(w_0)\eta=&d\pi(\gamma(S_2), \dot{\gamma}(S_2))\bigg(X(\gamma(S_2), \dot{\gamma}(S_2))\frac{\p \tau}{\p w}(y_0,w_0)\cdot \eta \\
&+ d \varphi_{\tau(y_0,w_0)} (y_0,w_0)(0,\eta)\bigg)\\
=&\dot{\gamma}(S_2)\frac{\p \tau}{\p w}(y_0,w_0)\cdot \eta + d\pi(\gamma(S_2), \dot{\gamma}(S_2))(d \varphi_{S_2} (y_0,w_0)(0,\eta)).
\end{aligned}
\end{equation}
Now by \cite[Lemma 1.40]{Paternain_book}, see also \cite[Theorem 11.2]{Ilmavirta_notes}, we get for the differential of the geodesic flow that 
\begin{equation}
\label{eq_0_D_8}
d \varphi_{S_2} (y_0,w_0)(0,\eta)=(J_{(0,\eta)}(S_2), \dot{J}_{(0,\eta)}(S_2)),
\end{equation}
where $J_{(0,\eta)}$ is the Jacobi field along the geodesic $t\mapsto \pi(\varphi_t(y_0,w_0))=\gamma(t)$ with the initial conditions
\begin{equation}
\label{eq_0_D_9}
J_{(0,\eta)}(0)=0,\quad \dot{J}_{(0,\eta)}=\eta. 
\end{equation}
Using \cite[Exercise 5.9]{Ilmavirta_notes}, \eqref{eq_0_D_9}, and the fact that $\eta\in T_{w_0}S_{y_0}M_0$,  we have 
\begin{equation}
\label{eq_0_D_10}
\langle \dot\gamma(S_2), J_{(0,\eta)}(S_2)\rangle = \langle \dot\gamma(0), J_{(0,\eta)}(0)\rangle +S_2 \langle \dot\gamma(0),\dot J_{(0,\eta)}(0)\rangle = S_2\langle w_0, \eta\rangle =0,
\end{equation}
showing that the Jacobi field $J_{(0,\eta)}$ is normal to $\gamma$. It follows from \eqref{eq_0_D_7} and \eqref{eq_0_D_8} that 
\begin{equation}
\label{eq_0_D_11}
F'_{y_0}(w_0)\eta= \dot{\gamma}(S_2)\frac{\p \tau}{\p w}(y_0,w_0)\cdot \eta +J_{(0,\eta)}(S_2). 
\end{equation}
Using  \eqref{eq_0_D_11} and the orthogonally \eqref{eq_0_D_10}, we see that  if $F'_{y_0}(w_0)$ has a non-trivial kernel, then there exists $\eta\ne 0$ such  $J_{(0,\eta)}(S_2)=0$, and therefore, the points $y_0$ and $\gamma(S_2)$ are conjugate points along $\gamma$, see \cite[Definition 7.3]{Ilmavirta_notes}. Thus, $F'_{y_0}(w_0)$ is bijective as long as $y_0$ is not a conjugate point to $\gamma(S_2)$  along $\gamma$. 

By the inverse function theorem,  $F_{y_0}$ is a local diffeomorphism if $y_0$ is not a conjugate point to $\gamma(S_2)$  along $\gamma$.  

Hence,  if $y_0$ is not a conjugate point to $\gamma(S_2)$ and $\gamma(-S_1)$  along $\gamma$, there exists a 
small neighborhood $W\subset S_{y_0}M_0$ of $w_0$ such that for every $w\in W$, $w \neq w_0$,  the unit speed geodesic $\eta:[-T_1,T_2]\to M_0$, $0<T_1,T_2<\infty$, such that $\eta(0)=y_0$ and $\dot{\eta}(0)=w$  is also non-tangential between boundary points, and $\gamma$ and $\eta$ do not  intersect each other at the boundary of $M_0$. Using the fact that $\gamma$ can only self-intersect at $y_0$ finitely many times, see \cite[Lemma 7.2]{Kenig_Salo_APDE_2013}, by choosing $W$ sufficiently small so that the corresponding finitely many tangent vectors of $\gamma$ and their negatives do not belong to $W$, we achieve that the geodesic $\eta$   and $\gamma$ are distinct and are not reverses of each other.  

To conclude the proof, we recall from  \cite[page 248]{do_Carmo} that $\{p\in \gamma([-S_1,S_2]): p \text{ is conjugate to }\gamma(-S_1) \text{ or }\gamma(S_2)\}$ is discrete, and since $M_0$ is compact, it is finite.  This completes the proof of the claim. 
\end{proof}

When proving Proposition \ref{prop_density_form} in the simplified setting, we shall need some basic facts about non-tangential geodesics. These facts are known, see \cite[Section 3]{DKuLLS_2018}, and are presented here for completeness and convenience of the reader. 

\begin{prop} 
\label{prop_D_geodesics}
Let $(M_0,g_0)$ be a smooth compact Riemannian manifold of dimension $n\ge 2$ with smooth boundary. 
\begin{itemize}
\item[(i)] Let $\gamma$ be a unit speed non-self-intersecting non-tangential geodesic on $M_0$, and let $y_0=\gamma(t_0)\in M_0^{\text{int}}$. Then there exists a small neighborhood $W$ of $w_0=\dot{\gamma}(t_0)$ in $S_{y_0}M_0$ such that for every $w\in W$, the unit speed geodesic $\gamma_{y_0,w}$ passing through  $(y_0, w)$ is non-tangential between boundary points and does not have self-intesections.

\item[(ii)] Let $\gamma$ and $\eta$ be unit speed non-self-intersecting non-tangential geodesics on $M_0$ with the only point of intersection $y_0=\gamma(t_0)=\eta(s_0)\in M_0^{\text{int}}$. Then there exists a small neighborhood $W$ of $w_0=\dot{\gamma}(t_0)$ in $S_{y_0}M_0$ such that for every $w\in W$, the unit speed geodesic $\gamma_{y_0,w}$ passing through  $(y_0,w)$ is non-tangential between boundary points, does not have self-intesections,  and intersects $\eta$ at the point $y_0$ only.

\end{itemize}

\end{prop}

\begin{proof}
Here we follow  \cite[Section 3]{DKuLLS_2018}. Let us prove (i).  Reparametrizing the geodesic $\gamma$ if necessary, we may assume that $\gamma:[-S_1,S_2]\to M_0$, $0<S_1,S_2<\infty$, is such that $\gamma(0)=y_0$ and $\dot{\gamma}(0)=w_0$. First the property of a geodesic being non-tangential is stable under small perturbations of the initial conditions, in view of $C^\infty$--dependence of the geodesic flow on the initial conditions.  Assume the contrary: there is a sequence $w_k\to w_0$ in  $S_{y_0}M_0$ as $k\to \infty$ such that there are times $t_k<s_k$ when the corresponding geodesic $\gamma_{y_0,w_k}:[-S_1(k),S_2(k)]\to M_0$ with $\gamma_{y_0,w_k}(0)=y_0$, $\dot{\gamma}_{y_0,w_k}(0)=w_k$, intersects itself, 
\begin{equation}
\label{eq_D_1}
a_k:=\gamma_{y_0,w_k}(t_k)=\gamma_{y_0,w_k}(s_k). 
\end{equation}
Note that the sequences $-S_1(k)\to -S_1$ and $S_2(k)\to S_2$ as $k\to \infty$. Therefore, 
the sequences $t_k, s_k$ are bounded, and passing to subsequences, we may assume that $t_k\to t_0$ and $s_k\to s_0$. Letting $k\to \infty$ in \eqref{eq_D_1}, we get $\gamma(t_0)=\gamma(s_0)$.
Since $\gamma$ does not have self-intersections we obtain that $t_0=s_0$. 

As all geodesics $\gamma_{y_0,w_k}$ are non-tangential,  it follows from \eqref{eq_D_1} that $a_k\in M_0^{\text{int}}$. As $M_0$ is compact, it has a positive injectivity radius $\text{Inj}(M_0)>0$. Here we have extended $M_0$ to a closed
manifold to speak about the injectivity radius and the boundary will not cause any problems as $a_k\in M_0^{\text{int}}$.  Now \eqref{eq_D_1} implies that 
\[
s_k\ge t_k+2 \text{Inj}(M_0),
\]
and therefore, $s_0-t_0\ge 2 \text{Inj}(M_0)>0$, which is a contradiction. Hence, (i) follows. 

To prove (ii), first reparametrizing the geodesics $\gamma$ and $\eta$ if necessary, we may assume that $\gamma:[-S_1,S_2]\to M_0$, $0<S_1,S_2<\infty$, is such that $\gamma(0)=y_0$ and $\dot{\gamma}(0)=w_0$, and $\eta:[-T_1,T_2]\to M_0$, $0<T_1,T_2<\infty$, is such that $\eta(0)=y_0$.  By (i), there exists a small neighborhood $W$ of $w_0$ in $S_{y_0}M_0$ such that for every $w\in W$, the unit speed geodesic $\gamma_{y_0,w}$ such that $\gamma_{y_0,w}(0)=y_0$ and $\dot{\gamma}_{y_0,w}(0)=w$ is non-tangential between boundary points and does not have self-intesections.  We shall show that the neighborhood $W$ can be made smaller so that every $\gamma_{y_0,w}$  intersects $\eta$ at the point $y_0$ only. Let us assume the opposite:  there is a sequence $w_k\to w_0$ in  $S_{y_0}M_0$ as $k\to \infty$ such that there are times $t_k\ne 0$, $s_k\ne 0$ when the corresponding geodesic $\gamma_{y_0,w_k}$ intersects $\eta$, 
\begin{equation}
\label{eq_D_2}
\gamma_{y_0,w_k}(t_k)=\eta(s_k).
\end{equation}
Note that here we used that  $\gamma_{y_0,w_k}$ and $\eta$ do not have self-intersections. We also have 
\begin{equation}
\label{eq_D_3}
\gamma_{y_0,w_k}(0)=\eta(0)=y_0.
\end{equation}
Passing to subsequences, we have that $t_k\to  t_0$ and $s_k\to s_0$. Thus, it follows from \eqref{eq_D_2} that 
$\gamma( t_0)=\eta(s_0)$,
and therefore,  as $\gamma$ and $\eta$ do not  self-intersect and $y_0$ is the only point of their intersection, we get $t_0=s_0=0$. In view of \eqref{eq_D_2} we have $\gamma_{y_0,w_k}(t_k)=\eta(s_k)\to \eta(0)=y_0\in M_0^{\text{int}}$, and thus, for $k$ sufficiently large, $\gamma_{y_0,w_k}(t_k)=\eta(s_k)\in M_0^{\text{int}}$. This together with 
\eqref{eq_D_3} gives $|t_k|>\text{Inj}(M_0)>0$ and $|s_k|>\text{Inj}(M_0)>0$ for $k$ sufficiently large, otherwise  the geodesics $\gamma_{y_0,w_k}$ and $\eta$ 
would intersect at a geodesic ball centered at  $y_0$, which is a contradiction.  Thus, (ii) follows. 
\end{proof}

\end{appendix}

\section*{Acknowledgements}
K.K. is deeply grateful to Gabriel Paternain for the generous help with the proof of Proposition \ref{prop_geodesic}. K.K. would also like to thank Ali Feizmohammadi and Lauri Oksanen for helpful discussions, and Giovanni Alessandrini for pointing out the reference \cite{Aless_deHoop_Gaburro_Sincich_2018}. The research of K.K. is partially supported by the National Science Foundation (DMS 1815922). The research of G.U. is partially supported by NSF, a Walker Professorship at UW and a Si-Yuan Professorship at IAS, HKUST. Part of the work was supported by the NSF grant DMS-1440140 while K.K. and G.U. were in residence at MSRI in Berkeley, California, during Fall 2019 semester.

\end{document}